\documentclass{amsart}

\usepackage{amscd}
\usepackage{amssymb}

\usepackage[all, knot, all, 2cell]{xy}
\usepackage{extpfeil}
\usepackage{hyperref}
\usepackage{color}
\usepackage{setspace}
\usepackage{MnSymbol}
\usepackage{mathrsfs}
\usepackage{pgfplots}
\usepackage{enumitem}

\usepgfplotslibrary{polar}
\usepgflibrary{shapes.geometric}
\usetikzlibrary{calc}

\pgfplotsset{soldot/.style={color=blue,only marks,mark=*}} \pgfplotsset{holdot/.style={color=blue,fill=white,only marks,mark=*}}

\def\tl{\dashv}

\def\cA{\mathcal A}
\def\cB{\mathcal B}
\def\cC{\mathcal C}

\def\cE{\mathcal E}
\def\cF{\mathcal F}
\def\cG{\mathcal G}

\def\cI{\mathcal I}
\def\cJ{\mathcal J}
\def\cK{\mathcal K}
\def\cL{\mathcal L}

\def\cN{\mathcal N}

\def\cP{\mathcal P}
\def\cQ{\mathcal Q}

\def\cS{\mathcal S}
\def\cT{\mathcal T}
\def\cU{\mathcal U}

\def\cX{\mathcal X}

\newcommand{\xra}[1]{\xrightarrow{#1}}

\newcommand{\vp}[1]{\vspace{#1in}}
\newcommand{\hp}[1]{\hspace{#1in}}

\def\sdarrow{\rightleftarrows}

\numberwithin{equation}{section}
\theoremstyle{plain} 
\newtheorem{thm}[equation]{Theorem}

\newtheorem*{introthm*}{Theorem}

\newtheorem{cor}[equation]{Corollary}
\newtheorem{lem}[equation]{Lemma}
\newtheorem{prop}[equation]{Proposition}
\newtheorem{lem-defn}[equation]{Lemma-Definition}

\theoremstyle{definition}
\newtheorem{defn}[equation]{Definition}

\newtheorem{ex}[equation]{Example}

\theoremstyle{remark}
\newtheorem{rem}[equation]{Remark}

\begin{document}

\title{Some topological aspects of a general spectra construction of Matsui and Takahashi}

\author{Xuan Yu}

\maketitle

\begin{abstract}

Matsui and Takahashi introduce a general spectra construction for triangulated categories in \cite{matsui2020construction}, which is later used to establish Matsui's theory of triangular geometry. In this paper, we study several topological aspects of this general construction and give criteria for soberness and spectralness of the spectra. Furthermore, we discuss and generalize the immersion phenomenon for Noetherian schemes as appeared in \cite{matsui2021prime}. The last section illustrates that similar immersions also appear if the underlying category has a well-defined finite group action. We work in the extriangulated context to incorporate similar ideas from the triangulated and abelian contexts.

\end{abstract}

\smallskip

\hp{0.25} \footnotesize\emph{Keywords: spectrum, support, immersion, group action}

\section{Introduction}

Subcategory classification problems are critical to understanding the categorical structure of a given category and closely tied to the ``geometry" of the underlying category. The first major achievement in this direction is Gabriel's reconstruction of a Noetherian scheme from its abelian category of quasi-coherent sheaves \cite{gabriel1962thesis}. For triangulated categories, there are the celebrated Bondal-Orlov reconstruction Theorem \cite{bondal2001reconstruction} and theory of tensor triangular geometry \cite{balmer2005spectrum}. These works require some extra assumptions of the underlying schemes or triangulated categories, and a recent breakthrough is Matsui's theory of triangular geometry \cite{hirano2022prime, ito2023gluing, ito2025polar, ito2024new, matsui2021prime, matsui2023triangular}, where the existence of a tensor product of the underlying triangulated category is loosened.

To begin with, an idea first appeared and investigated systematically in \cite{matsui2020construction} is to choose an appropriate class of thick subcategories (the ``prime" subcategories) $C$ and view them as the ``points" of the ``spectrum" $Esp_C(\cE)$ of the underlying category $\cE$. A key observation of Matsui is that if $C$ is chosen to be the collection of the so-called \textbf{prime thick subcategories} (Definition 2.8 \cite{matsui2023triangular}), then the corresponding spectrum do carry lots of important geometric information of the category. For example, there is an immersion \[X\hookrightarrow MSpc(Perf(X))\] from a Noetherian scheme $X$ to the \textbf{Matsui spectrum} of its category of perfect complexes \cite{matsui2021prime}, which is later shown to be an open immersion of ringed spaces in \cite{matsui2023triangular}. Moreover, Ito and Matsui show in \cite{ito2024new} that one can reproduce the Ballard-Bondal-Orlov reconstruction Theorem use the theory of triangular geometry. For more geometric results, please see \cite{hirano2022prime, ito2023gluing, ito2025polar}.

The purpose of the current paper is to discuss some topological aspects of the above spectra construction. First, to incorporate more types of subcategories and to get an actual topological space (with least requirements on the chosen subcategories) from the construction, we generalize Matsui and Takahashi's construction a little bit and state some general results about the corresponding support theory. In particular, the underlying categories in the current paper are extriangulated categories because we want to unify spectra constructions for both abelian and triangulated categories. We give criteria for $Esp_C(\cE)$ to be sober (Theorem \ref{sober}) and spectral (Theorem \ref{noethspectral} and Corollary \ref{spectral}), respectively. These properties are crucial to the spectra and has many important consequences. Our criteria are motivated by the work of Rowe \cite{rowe2024noncommutative} and Finocchiaro \cite{finocchiaro2014spectral} respectively, and they are sometimes quite convenient to check. For example, one can show easily that the Serre spectrum of an abelian category (Example \ref{serre}), Balmer spectrum of a tensor triangulated category (Example \ref{ex}) and the noncommutative Balmer spectrum of a monoidal triangulated category satisfying the tensor product property (Theorem \ref{ncspectral}) are spectral. The spectral criterion is also used in our generalization of the above Matsui's immersion result for Noetherian schemes in Section \ref{s5}.

Next, we consider the immersion and (topological) reconstruction problem for the spectra construction. Our discussion is based on a detailed study of the reconstruction map in tensor triangular geometry, i.e., the universal map $X\xra{}BSpc(\cT)$ that sends $x\in X$ to $\{T\in\cT: x\notin\sigma(T)\}$ from any support data $(X,\sigma)$ to the Balmer spectrum and support of a tensor triangulated category $\cT$. This map is also the one used in \cite{matsui2021prime} for the immersion/reconstruction for Matsui spectrum of triangulated categories. In fact, for any support $(X,\sigma)$ (Definition \ref{defsupp}) of a given extriangulated category $\cE$, the image $\{E\in\cE: x\notin\sigma(E)\}$ has two different descriptions, depending on the open/closedness of the support $\sigma(E)$, for any object $E$ in $\cE$. We discuss closed support in Section \ref{close}, open support in Section \ref{s5} and generalize Matsui's immersion/reconstruction Theorems. Following is a summary of our results: let $(X,\sigma)$ be a support on $\cE$, there are natural inclusion preserving maps \[\{\text{certain type subcategories of }\cE\}\overset{f_{\sigma}}{\underset{g_{\sigma}} \sdarrow} \{\text{certain type subsets of } X\}\] where $f_{\sigma}(\cN):=\bigcup\limits_{N\in\cN}\sigma(N)$ for a given type of subcategory $\cN$ and $g_{\sigma}(S):=\{E\in\cE: \sigma(E)\subseteq S\}$ for certain subset $S$. For any $x,y\in X$, we have \vp{0.1} 
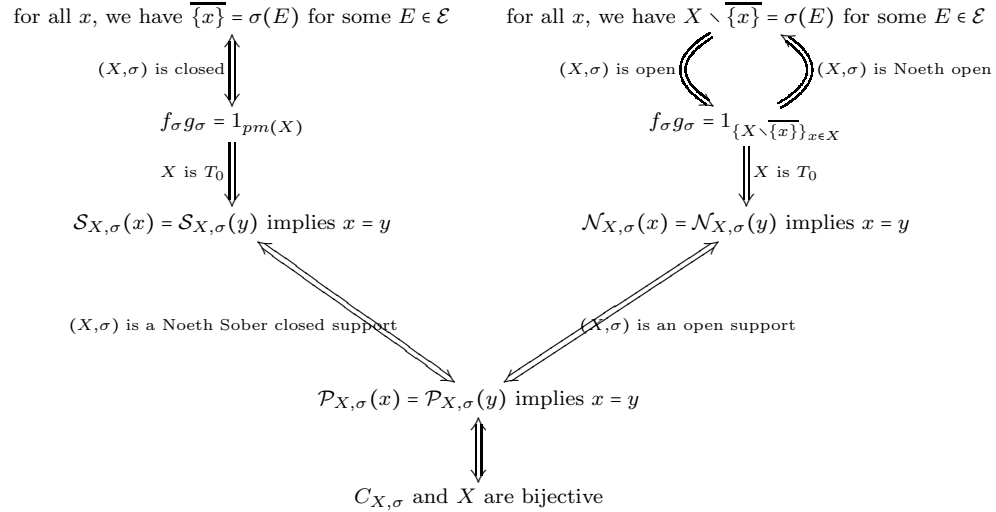
\[\xymatrix@C=-6.5em{\text{for all $x$, we have }\overline{\{x\}}=\sigma(E) \text{ for some } E\in\cE \ar@2{<->}[d]_{(X,\sigma)\text{ is closed}} && \text{for all $x$, we have } X\setminus\overline{\{x\}}=\sigma(E) \text{ for some } E\in\cE\ar@{=>}@/_2pc/[d]_{(X,\sigma)\text{ is open}} \\ f_{\sigma}g_{\sigma}=1_{pm(X)}\ar@{=>}[d]_{X \text{ is } T_0} && f_{\sigma}g_{\sigma}= 1_{\{X\setminus\overline{\{x\}}\}_{x\in X}}\ar@{=>}[d]^{X \text{ is } T_0}\ar@{=>}@/_2pc/[u]_{(X,\sigma)\text{ is Noeth open}} \\ \cS_{X,\sigma}(x)=\cS_{X,\sigma}(y) \text{ implies } x=y \ar@2{<->}[rdd]_{(X,\sigma) \text{ is a Noeth Sober closed support}} && \cN_{X,\sigma}(x)=\cN_{X,\sigma}(y) \text{ implies } x=y \ar@2{<->}[ldd]^{(X,\sigma) \text{ is an open support}} \\ & \\ & \cP_{X,\sigma}(x)=\cP_{X,\sigma}(y) \text{ implies } x=y\ar@2{<->}[d] & \\ & C_{X,\sigma}\text{ and }X\text{ are bijective} & }\] (please refer to the corresponding Sections for the notations) and we give immersion/reconstruction results for closed (Theorem \ref{immersion}) and open supports (Theorem \ref{opimmersion}) under slightly different hypotheses because of the above differences.

The immersion/reconstruction phenomenon also appears in other context(s). For example, in the last Section, we exhibit this for categories with a finite group action. We show, if the equivariantization $\cE^G$ of an extriangulated category $\cE$ carries a naturally induced extriangulated category structure, then there is an immersion from an action-invariant spectrum of $\cE$ to an appropriately chosen spectrum of $\cE^G$ (Theorem \ref{actionthm}). In particular, this gives an immersion of the $G$-Matsui spectrum of $\cE$ to the Matsui spectrum of $\cE^G$ (Corollary \ref{gmatimmer}) if $\cE$ is a closed tensor triangulated category generated by the tensor unit. Similarly, for an abelian category $\cA$ with a group action by a finite group $G$, there is an immersion from the $G$-Serre spectrum to the Serre spectrum of $\cA^G$ (Corollary \ref{gserreimmer}).

We use $\cE$, $\cT$, $\cA$ and $\cG$ to denote an extriangulated, triangulated, abelian and Grothendieck category respectively throughout the paper.

\vp{0.3}

\section{Topological Preliminaries} \label{sec2}

We begin with a few basic but important facts from (spectral) topology.

\begin{lem} \label{top} Let $U\subseteq X$ be an open set in a topological space $X$ and $x\in X$ is a point in the space, then \begin{enumerate} \item $U\cap\overline{\{x\}}=\emptyset$ if and only if $x\notin U$. \item $U\cap\overline{\{x\}}=\emptyset$ implies $y\notin U$ for any $y\in\overline{\{x\}}$. \end{enumerate} \end{lem}

\begin{proof} For the first statement, we have

$(\Rightarrow)$: if not, that is, $x\in U$, then $x\in U\cap\overline{\{x\}}$, contradiction.

$(\Leftarrow)$: if not, that is, there is some $y\in X$ such that $y\in U$ and $y\in\overline{\{x\}}$. Note that open subsets are generalization-closed, this implies $x\in U$, contradiction.

For any $y\in\overline{\{x\}}$, we have $\overline{\{y\}}\subseteq\overline{\{x\}}$, so $U\cap\overline{\{y\}}=\emptyset$ and this gives $y\notin U$ by $(1)$. \end{proof}

\begin{defn}

A topological space $X$ is \textbf{spectral} if it's homeomorphic to the spectrum of a commutative ring. The \textbf{Hochster dual} $X^{\vee}$ is a new topology on $X$ by taking as basic open subsets the closed sets with quasi-compact complements. \end{defn}

Note that the "closed" and "quasi-compact" in the above Definition both refer to the original topology of $X$ and they are the \textbf{Thomason subsets} of $X$. 

\begin{lem} \label{nsspace} Let $X$ be a Noetherian spectral space and $U$ be an open subset in $X$. Then

\begin{enumerate}

\item $U$ is quasi-compact in $X$. Hence, the basic open subsets of the Hochster dual $X^{\vee}$ are simply the closed sets of $X$.
\item $U$ is closed in $X^{\vee}$.

\end{enumerate}

In particular, open sets in $X^{\vee}$ are unions of closed subsets of $X$ (that is, the specialization-closed subsets of $X$), and closed sets of $X^{\vee}$ are intersections of open subsets of $X$ (that is, the generalization-closed subsets of $X$). \end{lem}

\begin{proof} The first statement is a standard fact for Noetherian spaces. For the second, it's enough to show $X\setminus U$ is open in $X^{\vee}$. By definition, this amounts to show $X\setminus U$ can be written as a union of closed sets in $X$. This is obvious as $X\setminus U$ is itself closed in $X$. \end{proof}

\begin{ex} $X=Spec(R)$ is a Noetherian spectral space, when $R$ is a Noetherian commutative ring. \end{ex}

\begin{ex} Any Noetherian scheme is a Noetherian spectral space.
\end{ex}

\begin{ex} In general, any Noetherian Kolmogorov space is Noetherian spectral. (A space is Kolmogorov if for any two distinct points there exists an open subset containing exactly one of the two). \end{ex}

\vp{0.3}


\section{Spectrum and support for extriangulated categories}

\subsection{Spectrum and support} \label{supp} Let $\cE$ be an extriangulated category. For references for extriangulated categories, one can see \cite{nakaoka2019extri}.

\begin{defn} Let $\cN$ be a non-empty full additive subcategory of $\cE$ and $L\xra{}M\xra{}N$ a conflation, we say

\begin{enumerate}

\item $\cN$ is \textbf{thick}, if it's closed under direct summands and has the 2-out-of-3 property for conflations, that is, if two of three objects $L,M,N$ belong to $\cN$, then so does the third. Denote the collection of thick subcategories by $Th(\cE)$.

\item $\cN$ is \textbf{Serre}, if $L, N\in\cN$ if and only if $M\in\cN$. Denote the collection of Serre subcategories by $Serre(\cE)$.

\end{enumerate}

Denote the collection of non-empty subcategories closed under finite direct sums and direct summands (hence containing the zero object) by $as(\cE)$ (i.e., ``add"+``smd"). It is clear that $Serre(\cE)\subseteq Th(\cE)\subseteq as(\cE)$. \end{defn}
It is obvious that the above definitions reduce to thick and Serre subcategories in the usual sense when $\cE$ is in fact triangulated and abelian, respectively. Following is motivated by the definitions in \cite{balmer2005spectrum} \cite{matsui2019singular} and \cite{yu2023support}.

\begin{defn}\label{defsupp} A \textbf{support} for $\cE$ is a pair $(X,\sigma)$ where $X$ is a topological space and an assignment $\sigma: \cE\xra{} subset(X)$ such that $\sigma(0)=\emptyset$ and $\sigma(E\oplus F)=\sigma(E)\cup \sigma(F)$ for any $E,F\in\cE$. We say $(X,\sigma)$ is a \textbf{Noetherian/spectral support} if $X$ is Noetherian/spectral. Similarly, it's an \textbf{open/closed support} if $\sigma(E)$ is open/closed for any object $E\in\cE$. \end{defn}

\begin{defn}
Let $(X,\sigma)$ and $(Y,\tau)$ be two supports on the same category $\cE$. A \textbf{morphism of supports} $f: (X,\sigma)\xra{}(Y,\tau)$ is a continuous map $f:X\xra{}Y$ such that $\sigma(M)=f^{-1}(\tau(E))$ for any $E\in\cE$. Such a morphism is an isomorphism if and only if $f$ is a homeomorphism.
\end{defn}

We generalize the spectra construction of \cite{matsui2020construction} to extriangulated categories to the following.

\begin{defn} (\cite{matsui2020construction}) Suppose $C$ is a subset of $as(\cE)$ and $\cX$ a class of objects, define \[Z_C(\cX):=\{\cP\in C: \cP\cap\cX=\emptyset\}\] \[U_C(\cX):=C\setminus Z_C(\cX)=\{\cP\in C: \cP\cap\cX\neq\emptyset\}\] \end{defn}

One sees immediately that

\begin{itemize}
  \item $Z_C(\cE)=\emptyset$ and $Z_C(\emptyset)=C$
  \item $\bigcap_{i\in I} Z_C(\cX_i)=Z_C(\bigcup_{i\in I}\cX_i)$ for any family $\{\cX_i\}_{i\in I}$ of objects of $\cE$.
  \item $Z_C(\cX)\cup Z_C(\cX')=Z_C(\cX\oplus\cX')$ for classes of objects $\cX,\cX'$, where $\cX\oplus\cX':=\{E\oplus E': E\in\cX, E'\in\cX'\}$.
\end{itemize}

This gives a topology on $C$ by defining the collection of closed subsets to be $\{Z_C(\cX)\}$ (the open subsets are $\{U_C(\cX)\}$), denote this topological space by $Esp_C(\cE)$ and call it the \textbf{C-spectrum} of $\cE$.

\begin{defn} (\cite{matsui2020construction}) For any object $E\in\cE$, its \textbf{C-support} is defined to be \[supp_C(E):=Z_C(\{E\})=\{\cP\in Esp_C(\cE): E\notin\cP\}\] \end{defn} We have

\begin{lem} \label{blem} \hp{1}

\begin{enumerate}

\item $\{supp_C(E)\}_{E\in\cE}$ is a closed basis for $Esp_C(\cE)$. Similarly, $\{U_C(E)\}_{E\in\cE}$ is an open basis for $Esp_C(\cE)$.

\item $supp_C(0)=\emptyset$ and $supp_C(E\oplus F)=supp_C(E)\cup supp_C(F)$ for any $E,F\in\cE$, that is, $(Esp_C(\cE),supp_C)$ is a closed support on $\cE$, for any choice of $C\subseteq as(\cE)$. \end{enumerate}

Moreover, \begin{enumerate}\setcounter{enumi}{2}

\item If $C\subseteq Th(\cE)$. Then for any conflation $L\xra{}M\xra{}N$ we have \[supp_C(M)\subseteq supp_C(L)\cup supp_C(N)\] \[supp_C(L)\subseteq supp_C(M)\cup supp_C(N)\] \[supp_C(N)\subseteq supp_C(L)\cup supp_C(M)\]

\item If $C\subseteq Serre(\cE)$. Then for any conflation $L\xra{}M\xra{}N$ we have \[supp_C(M)=supp_C(L)\cup supp_C(N)\]

\end{enumerate} \end{lem}

\begin{proof} The first statement holds because $Z_C(\cX)=\bigcap\limits_{E\in\cX}supp_C(M)$ for any $\cX$ and the rest are straightforward. \end{proof}

The reason for asking $C$ to be a subset of $as(\cE)$ is because this is what makes $Esp_C(\cE)$ into an actual topological space and $supp_C$ a closed support. Also,

\begin{lem} For any $C\subseteq as(\cE)$ and $\cN\in C$, we have $\bigcup_{N\in\cN} supp_C(N)=\{\cP\in C: \cN\nsubseteq\cP\}$. \end{lem}

\begin{proof} Indeed, $\cP\in \bigcup_{N\in\cN} supp_C(N)$ means $\cP\in supp_C(N)$ for some $N\in\cN$, that is, there is some $N\in\cN$ such that $N\notin\cP$ so $\cN\nsubseteq\cP$. Conversely, if there is some $N\in\cN$ such that $N\neq\cP$, this means exactly $\cP\in supp_C(N)$. \end{proof}

\begin{prop} \label{closure} Let $S\subseteq Esp_C(\cE)$. The closure of $S$ is $\overline{S}=\bigcap_{S\subseteq supp_C(E)} supp_C(E)$. \end{prop}

\begin{proof} This follows from the fact that $\{supp_C(E)\}_{E\in\cE}$ is a closed basis for the topology of $Esp_C(\cE)$. Indeed, $\bigcap_{S\subseteq supp_C(E)} supp_C(E)=Z_C(\{E\in\cE: S\subseteq supp_C(E)\})$ so it's closed and contains $S$. Let $T$ be any closed subset containing $S$, $T$ is the intersection of $supp_C(E)$ for some $E\in\cE$ since $\{supp_C(E)\}_{E\in\cE}$ is a closed basis. Then $S\subseteq\bigcap\limits_{\text{some } E} supp_C(E)$ forces $S\subseteq supp_C(E)$ for all the objects in the collection. \end{proof}

The proof in \cite{matsui2020construction} also gives

\begin{prop}\label{t0} (Proposition 2.3 \cite{matsui2020construction}) For any $\cP\in C$, one has $\overline{\{\cP\}}=\{\cQ\in Esp_C(\cE): \cQ\subseteq\cP\}$. In particular, $Esp_C(\cE)$ is a $T_0$-space. \end{prop}

Suppose $C_1\subseteq C_2\subseteq as(\cE)$ are subsets of $as(\cE)$, we obviously have

\begin{lem} \label{two} There is an immersion $Esp_1(\cE)\hookrightarrow Esp_2(\cE)$ where $Esp_i(\cE):=Esp_{C_i}(\cE)$. \end{lem}

\subsection{Morphisms and idempotent completion for spectra of triangulated categories} Suppose $\cT$ is a triangulated category and $\cT^{\natural}$ its idempotent completion. In this subsection, we consider subsets $C\subseteq Th(\cT)$, although some of the followings hold under more general assumptions. The purpose of this subsection is to show there is an appropriate subset $D\subseteq Th(\cT^{\natural})$ such that $Esp_D(\cT^{\natural})$ is homeomorphic to $Esp_C(\cT)$, so one may assume the underlying category $\cT$ to be idempotent complete when necessary.

Let $F:\cT'\xra{}\cT$ be a triangulated functor and $D\subseteq Th(\cT)$ a subset, define \[F^{-1}D:=\{F^{-1}\cP: \cP\in D\}\] It's straightforward to show $F^{-1}\cP\in Th(\cT')$ so $F^{-1}D\subseteq Th(\cT')$. Define \[Esp\,F: Esp_D\cT\xra{}Esp_{F^{-1}D}\cT'\] \[\hp{0.35}\cP\mapsto F^{-1}\cP\]

\begin{lem} \label{cts} $Esp\,F$ is continuous. \end{lem}

\begin{proof} We will show $(Esp\,F)^{-1}(U_{F^{-1}D}(M))=U_D(F(M))$ for any $M\in\cT'$. Indeed, $LHS=(Esp\,F)^{-1}(\{F^{-1}\cP: \cP\in D \text{ and } M\in F^{-1}\cP\})=\{\cP\in D: F(M)\in\cP\}=RHS$. \end{proof}

A triangulated functor $F:\cT'\xra{}\cT$ is a \textbf{triangle equivalence up to direct summands} if it's fully faithful and for any object $M\in\cT$, there is another object $N\in\cT$ such that $M\oplus N\in F(\cT')$.

\begin{lem} \label{equisum} (Lemma 2.10 \cite{matsui2021prime}) Let $F:\cT'\xra{}\cT$ be a triangle equivalence up to direct summand. Then one has $M\oplus M[1]\in F(\cT')$ for any $M\in\cT$. \end{lem}

Let $F:\cT'\xra{}\cT$ be a triangle equivalence up to direct summands starting from now and $M\in\cT$ any object. By Lemma \ref{equisum} there is some object $M'\in\cT'$ such that $M\oplus M[1]\cong F(M')$.

\begin{lem} \label{open} For any $\cP\in Th(\cT)$, we have $M\in\cP$ if and only if $M'\in F^{-1}\cP$. \end{lem}

\begin{proof} For any $M\in\cP$, we have $F(M')\cong M\oplus M[1]\in\cP$ as $\cP$ is closed under finite direct sums and shifts, this gives $M'\in F^{-1}\cP$. Conversely, if $M'\in F^{-1}\cP$, i.e., $M\oplus M[1]\cong F(M')\in\cP$, then $M\in\cP$ as $\cP$ is closed under direct summands. \end{proof}

\begin{lem} \label{inj} Suppose $\cP'=F^{-1}\cP$ for some $\cP\in Th(\cT)$. Then $\cP=\{M\in\cT: M\oplus M[1]\in F(\cP')\}$. \end{lem}

\begin{proof} First of all, we always have $F(F^{-1}\cP)\subseteq\cP$ for any $\cP\in Th(\cT)$. Now, pick any $M\in\cP$, we have $M\oplus M[1]\cong F(M')$ for some $M'\in\cT'$ by Lemma \ref{equisum}, so $F(M')\in\cP$ as $\cP$ is a closed under finite direct sums and shifts, i.e., $M'\in F^{-1}\cP=\cP'$ and this means $M\oplus M[1]\cong F(M')\in F(\cP')$. Conversely, $M\oplus M[1]\in F(\cP')=F(F^{-1}\cP)\subseteq\cP$ implies $M\in\cP$ as $\cP$ is closed under direct summands. \end{proof}

\begin{thm} Let $F:\cT'\xra{}\cT$ be a triangle equivalence up to direct summands, then $Esp\,F: Esp_D(\cT)\xra{\cong} Esp_{F^{-1}D}(\cT')$ is a homeomorphism. \end{thm}

\begin{proof} Lemma \ref{cts} and \ref{inj} tell us $Esp\,F$ is continuous and bijective (surjectivity is obvious by construction). It remains to show (for example) $Esp\,F$ is open and this is Lemma \ref{open}. Indeed, for any $M\in\cT$, there is some $M'\in\cT'$ such that $M\oplus M[1]\cong F(M')$ so $(Esp\,F)(U_D(M))=U_{F^{-1}D}(M')$. \end{proof}

Let $\cT$ be a triangulated category. The idempotent completion of $\cT$, denoted by $\cT^{\natural}$, is a triangulated category whose objects are pairs $(M,e)$, where $M$ is an object of $\cT$ and $e:M\xra{}M$ is an idempotent. The shift functor in $\cT^{\natural}$ is $(M,e)[1]:=(M[1],e[1])$ and direct sum is given by $(M,e_M)\oplus (N,e_N):=(M\oplus N, e_M\oplus e_N)$. There is a natural fully faithful triangulated functor $\iota: \cT\xra{}\cT^{\natural}$ which sends an object $M$ to $(M,1_M)$ so we treat $\cT$ as a triangulated subcategory of $\cT^{\natural}$, and $\iota$ is a triangle equivalence up to direct summands. By above, we have

\begin{cor} \label{idem} There is a homeomorphism $Esp\,(\iota): Esp_D(\cT^{\natural})\xra{\cong} Esp_{\iota^{-1}D}(\cT)$ for any $D\subseteq Th(\cT^{\natural})$. \end{cor}

For any $C\subseteq Th(\cT)$, define $\cP^{\natural}:=\{(M,e)\in\cT^{\natural}: (M,e)\oplus (M,e)[1]\in\cP\}$ for any $\cP\in C$ and $C^{\natural}:=\{\cP^{\natural}: \cP\in C\}$. It is straightforward that

\begin{lem} \label{id} $\cP^{\natural}\in Th(\cT^{\natural})$ and $\cP=\iota^{-1}\cP^{\natural}=\cP^{\natural}\cap\cT$. \end{lem}

\begin{proof} For any $(M,e)\in\cP^{\natural}$ we have $(M,e)[1]\oplus (M,e)[2]\cong ((M,e)\oplus (M,e)[1])[1]\in\cP$ so $(M,e)[1]\in\cP^{\natural}$. Next, for any distinguished triangle $(L,e_L)\xra{}(M,e_M)\xra{}(N,e_N)\xra{}(L,e_L)[1]$ in $\cT^{\natural}$ with $(L,e_L), (N,e_N)\in \cP^{\natural}$, then the triangle $(L,e_L)[1]\xra{}(M,e_M)[1]\xra{}(N,e_N)[1] \xra{}((L,e_L)[1])[1]$ is also distinguished in $\cT^{\natural}$ so \[(L,e_L)\oplus (L,e_L)[1]\xra{}(M,e_M)\oplus (M,e_M)[1]\xra{}(N,e_N)\oplus (N,e_N)[1]\rightsquigarrow\] a distinguished triangle in $\cT^{\natural}$ with $(L,e_L)\oplus (L,e_L)[1]$ and $(N,e_N)\oplus (N,e_N)[1]$ in $\cP$ (hence in $\cT$) so it's in fact a distinguished triangle in $\cT$ with the first and third term in $\cP$. This implies $(M,e_M)\oplus (M,e_M)[1]\in\cP$ so $(M,e_M)\in\cP^{\natural}$. Finally, for any $(M,e_M)\oplus (N,e_N)\in\cP^{\natural}$, that is, $(M\oplus N,e_M\oplus e_N)\oplus (M\oplus N,e_M\oplus e_N)[1]\in\cP$ we have $(M\oplus M[1],e_M\oplus e_{M[1]}), (N\oplus N[1], e_N\oplus e_{N[1]})\in\cP$ so $\cP^{\natural}$ is closed under direct summands.

Now, $(X\oplus X[1],1_{X\oplus X[1]})\in\cP$ for any $X\in\cP$ as $\cP$ is closed under finite direct sums and shifts so $\cP\subseteq\cP^{\natural}\cap\cT$. Conversely, suppose $(X,1_X)\in\cP^{\natural}\cap\cT$, then $(X\oplus X[1],1_{X\oplus X[1]})\cong (X,1_X)\oplus (X,1_X)[1]\in\cP$ and this gives $X\oplus X[1]\in\cP$ so in turn $X\in\cP$ as $\cP$ is closed under direct summands. \end{proof}

Therefore,

\begin{cor} There is a homeomorphism $Esp(\iota): Esp_{C^{\natural}}(\cT^{\natural})\xra{\cong} Esp_C(\cT)$. \end{cor}

\begin{proof} By Corollary \ref{idem}, take $D=C^{\natural}$ and $C=\iota^{-1}C^{\natural}$ by Lemma \ref{id}. \end{proof}

\vp{0.3}

\section{Criteria}

We discuss soberness and spectralness of the spectra and provide criterion respectively in this section. A subset $C\subseteq as(\cE)$ is fixed throughout the section.

\subsection{Criterion for soberness} \label{scri} The criterion in this subsection is motivated by Proposition 3.5 \cite{rowe2024noncommutative}. Let $Z\subseteq Esp_C(\cE)$ be a nonempty closed subset, define $\cP(Z)=\{E\in\cE: U_C(E)\cap Z\neq\emptyset\}$, we have

\begin{lem} Let $Z\subseteq Esp_C(\cE)$ be a nonempty irreducible closed subset. Then

\begin{enumerate}

\item $\cP(Z)\in as(\cE)$ if $C\subseteq as(\cE)$.
\item $\cP(Z)\in Th(\cE)$ if $C\subseteq Th(\cE)$.

\end{enumerate}
\end{lem}

\begin{proof} Suppose $E,F\in\cP(Z)$. Then we have $E\oplus F\in\cP(Z)$. If not, $(U_C(E)\cap U_C(F))\cap Z=U_C(E\oplus F)\cap Z=\emptyset$ implies $U_C(E)\cap Z=\emptyset$ or $U_C(F)\cap Z=\emptyset$, as $Z$ is irreducible. To show $\cP(Z)$ is closed under direct summands, notice that $L\oplus M\in\cP$ means $U_C(L)\cap U_C(M)\cap Z=U_C(L\oplus M)\cap Z\neq\emptyset$ so $U_C(L)\cap Z\neq\emptyset$ and $U_C(M)\cap Z\neq\emptyset$, that is, $L,M\in\cP$.

Suppose $C\subseteq Th(\cE)$ and $E\xra{}F\xra{}G$ be a conflation with $E,F\in\cP(Z)$. Note that we have $G\in thick(E\oplus F)$ and this gives $U_C(E\oplus F)\subseteq U_C(G)$. Indeed, any $\cP\in U_C(E\oplus F)$ means $E\oplus F\in\cP$ and $\cP\in C\subseteq Th(\cE)$ is a thick subcategory. Hence $G\in thick(E\oplus F)\subseteq\cP$. From $U_C(E\oplus F)\cap Z\neq\emptyset$ we get $U_C(G)\cap Z\neq\emptyset$ so $G\in\cP(Z)$. The same argument also shows $E,G\in\cP(Z)$ (resp. $F,G\in\cP(Z)$) then $F\in\cP(Z)$ (resp. $E\in\cP(Z)$) so $\cP(Z)$ has the two-out-of-three property for conflations. \end{proof}

Below is a characterization of soberness of the space $Esp_C(\cE)$, for any $C\subseteq as(\cE)$.

\begin{thm} \label{sober} Let $C\subseteq as(\cE)$ be a subset. Then $Esp_C(\cE)$ is a sober space if and only if $\cP(Z)\in C$ for any nonempty irreducible closed subset $Z\subseteq Esp_C(\cE)$. In fact, when this holds, we have $Z=\overline{\{\cP(Z)\}}$. In particular, $Esp_{as(\cE)}(\cE)$ and $Esp_{Th(\cE)}(\cE)$ are always sober. \end{thm}

\begin{proof} $(\Leftarrow)$ We will show $Z=\overline{\{\cP(Z)\}}$ and this means every irreducible closed subset has a generic point. Because $Esp_C(\cE)$ is $T_0$ by Proposition \ref{t0}, the generic point is unique, and this means $Esp_C(\cE)$ is sober. Indeed, pick any $\cQ\in Z$ and $E\in\cQ$ an object, we have $\cQ\in U_C(E)\cap Z$ so $U_C(E)\cap Z\neq\emptyset$. By definition this means $E\in\cP(Z)$, so in fact $\cQ\subseteq\cP(Z)$ and hence $\cQ\in\overline{\{\cP(Z)\}}$ by Proposition \ref{t0}. This shows $Z\subseteq\overline{\{\cP(Z)\}}$. For the other direction, it's enough to show $\cP(Z)\in Z$. Note that there are some objects $E\in\cE$ such that $Z\subseteq supp_C(E)$ by Proposition \ref{closure}. This gives $U_C(E)\cap Z=\emptyset$, so $E\notin\cP(Z)$. Since $\cP(Z)\in C$ by assumption, together gives $\cP(Z)\in supp_C(E)$ for all such objects $E$. Hence, \[\cP(Z)\in \bigcap_{Z\subseteq supp_C(E)} supp_C(E)=\overline{Z}=Z \] $(\Rightarrow)$ Let $Z$ be a nonempty irreducible closed subset of $Esp_C(\cE)$, then $Z=\overline{\{\cP\}}$ for some $\cP$ (note that $\cP\in Z\subseteq C$). We will show $\cP=\cP(Z)$ so $\cP(Z)\in C$. Indeed, for any $E\in\cP(Z)$, that is, $U_C(E)\cap Z\neq\emptyset$, we have some $\cQ\in Z$ such that $\cQ\in U_C(E)$, so $E\in\cQ\in Z=\overline{\{\cP\}}$. By Proposition \ref{t0} there is some $\cQ'\subseteq\cP$ such that $E\in\cQ'$ so $\cP(Z)\subseteq\cP$. Conversely, for any $E\in\cP$, we have $U_C(E)\cap Z=U_C(E)\cap\overline{\{\cP\}}\neq\emptyset$ so $E\in\cP(Z)$ by Lemma \ref{top}. \end{proof}

The above is a generalization of Proposition 3.5 \cite{rowe2024noncommutative} for the noncommutative Balmer spectrum of monoidal triangulated categories.

\subsection{Criterion for spectralness} We provide two criteria for $Esp_C(\cE)$ to be a spectral space. The first is an immediate consequence of the soberness criterion discussed last subsection. To be precise, assume $Esp_C(\cE)$ is Noetherian, we can give a characterization of spectralness of the space $Esp_C(\cE)$:

\begin{thm} (cf. Theorem 3.7 \cite{rowe2024noncommutative}) \label{noethspectral} If $Esp_C(\cE)$ is Noetherian, then it is a spectral space if and only if it's sober, i.e., by Theorem \ref{sober}, if and only if $\cP(Z)\in C$ for any nonempty irreducible closed subset $Z\subseteq Esp_C(\cE)$. \end{thm}

\begin{proof} Exactly like the proof in \cite{rowe2024noncommutative}, combine with Theorem \ref{sober}. \end{proof}

This gives a generalization of Theorem 3.7 \cite{rowe2024noncommutative}. Next, we start our discussion about the second spectralness criterion, recall

\begin{defn} A nonempty collection $\cF$ of subsets of a given set $X$ is a \textbf{filter} if \begin{enumerate} \item $\emptyset\notin\cF$; \item $Y\cap Z\in\cF$ if $Y, Z\in\cF$; \item $Z\in\cF$ and $Z\subseteq Y\subseteq X$ implies $Y\in\cF$. \end{enumerate} A filter $\cF$ is an \textbf{ultrafilter} if for each subset $Y$ of $X$, we have either $Y\in\cF$ or $X\setminus Y\in\cF$. We denote ultrafilter by $\cU$. Note, in particular, $X\in\cU$ for any ultrafilter $\cU$. \end{defn}

\begin{thm} (Corollary 3.3 \cite{finocchiaro2014spectral}) \label{finocchiaro} Let $X$ be a topological space. Then $X$ is a spectral space if and only if $X$ is $T_0$ and there is a basis $\cB$ of $X$ such that \[\{x\in X: \forall B\in\cB, x\in B \Leftrightarrow B\in\cU\}\neq\emptyset\] for any ultrafilter $\cU$ on $X$. \end{thm}

Let $C\subseteq as(\cE)$ be a subset and $\cU$ an ultrafilter on $Esp_C(\cE)$, define $\cP_{\cU}:=\{E\in\cE: U_C(E)\in\cU\}$. Note by definition the zero object $0$ is contained in any subcategory in $as(\cE)$, so $0\in\cP\in C$ and $U_C(0)=Esp_C(\cE)\in\cU$, this means $0\in\cP_{\cU}$ so $\cP_{\cU}$ is always nonempty.

\begin{lem} \label{sppre} Suppose $C\subseteq as(\cE)$ and $\cU$ is an ultrafilter on $Esp_C(\cE)$. Then $\cP_{\cU}\in as(\cE)$. Moreover, $\cP_{\cU}\in Th(\cE)$ if $C\subseteq Th(\cE)$. \end{lem}

\begin{proof} Suppose $M,N\in\cP_{\cU}$, that is, $U_C(M), U_C(N)\in\cU$, then $U_C(M\oplus N)=U_C(M)\cap U_C(N)\in\cU$ and this gives $M\oplus N\in\cP_{\cU}$. Similarly, from $M\oplus N\in\cP_{\cU}$, that is, $U_C(M)\cap U_C(N)=U_C(M\oplus N)\in\cU$ we get $U_C(M), U_C(N)\in\cU$ so $M,N\in\cP_{\cU}$.

Next, let $L\xra{}M\xra{}N$ be a conflation and $L,N\in\cP_{\cU}$, then $U_C(L)\cap U_C(N)\in\cU$ as $U_C(L), U_C(N)$ belong to $\cU$, hence $U_C(M)\in\cU$ because $U_C(M)\supseteq U_C(L)\cap U_C(N)$ and $\cU$ is a filter. The other two containments can be proved similarly. \end{proof}

Therefore, we have the following criterion for spectralness:

\begin{cor} \label{spectral} Let $C\subseteq as(\cE)$ be a subset, the space $Esp_C(\cE)$ is spectral if $\cP_{\cU}\in C$ for any ultrafilter $\cU$ on $Esp_C(\cE)$. In particular, $Esp_{as(\cE)}(\cE)$ and $Esp_{Th(\cE)}(\cE)$ are always spectral. \end{cor}

\begin{proof} The space $Esp_C(\cE)$ is always $T_0$, so by Theorem \ref{finocchiaro}, it's enough to show, for any ultrafilter $\cU$ on $Esp_C(\cE)$, \[\cE_C(\cU):=\{\cP\in Esp_C(\cE): \forall U_C(E), \cP\in U_C(E)\Leftrightarrow U_C(E)\in\cU\}\neq\emptyset\] because $\{U_C(M)\}_{M\in\cE}$ is an open basis for $Esp_C(\cE)$. By Lemma \ref{sppre}, $\cP_{\cU}$ belongs to $\cE_C(\cU)$ and this completes the proof. \end{proof}

Now we discuss some examples of the criterion.

\begin{ex} \label{ex} Let $(\cT,\otimes,\mathbf{1})$ be a tensor triangulated category, consider $C_{ptt}$, the collection of all prime tt-ideals (i.e., thick tensor ideals), then the Balmer spectrum is spectral. Indeed, it's enough to show $\cP_{\cU}$ is a prime tt-ideal for any ultrafilter $\cU$ on the Balmer spectrum by Corollary \ref{spectral}. First of all $\cP_{\cU}$ is proper because $\mathbf{1}\notin\cP_{\cU}$ ($\mathbf{1}$ is not contained in any prime tt-ideal hence $U_{ptt}(\mathbf{1}):= U_{C_{ptt}}(\mathbf{1})=\emptyset \notin\cU$).

Now, suppose $E\in\cP_{\cU}$ and $T$ an arbitrary object of $\cT$, we have $U_{ptt}(E)\subseteq U_{ptt}(E)\cup U_{ptt}(T)=U_{ptt}(E\otimes T)$ and this gives $U_{ptt}(E\otimes T)\in\cU$, i.e., $E\otimes T\in\cP_{\cU}$. Next, $E\otimes F\in\cP_{\cU}$ implies at least one of $E$ and $F$ in $\cP_{\cU}$, otherwise $U_{ptt}(E), U_{ptt}(F)\notin\cU$ gives $BSpc(\cT)\setminus U_{ptt}(E), BSpc(\cT)\setminus U_{ptt}(F)$ in $\cU$ (since $\cU$ is an ultrafilter) so \[BSpc(\cT)\setminus\bigg(U_{ptt}(E)\cup U_{ptt}(F)\bigg)=\bigg(BSpc(\cT)\setminus U_{ptt}(E)\bigg)\cap \bigg(BSpc(\cT)\setminus U_{ptt}(F)\bigg)\in\cU\] hence $U_{ptt}(E\otimes F)=U_{ptt}(E)\cup U_{ptt}(F)\notin\cU$ gives a contradiction. \end{ex}

\begin{ex} \label{serre} Let $\cA$ be an abelian category, consider $C_{Serre}$, the collection of all Serre subcategories, then the Serre spectrum $Esp_{Serre}(\cA)=Esp_{C_{Serre}}(\cA)$ is spectral. As above, it's enough to show $\cP_{\cU}$ is a Serre subcategory for any ultrafilter $\cU$ on the space $Esp_{Serre}(\cA)$ by Corollary \ref{spectral}. In this situation, the Serre-support $supp_{Serre}$ satisfies $supp_{Serre}(M)=supp_{Serre}(L)\cup supp_{Serre}(N)$ for any short exact sequence $0\xra{}L\xra{}M\xra{}N\xra{}0$ and this property gives what we want. \end{ex}

Now, we use the criterion to show the noncommutative Balmer spectrum (in the sense of \cite{nakano2022noncommutative}) for certain monoidal triangulated category is spectral. Recall a thick subcategory $\cP$ is a \textbf{thick two-sided ideal} if it's closed under left and right tensoring with any object of the underlying category. It is a \textbf{nc-prime} if it's proper, and \[\text{For objects $A,B\in\cK$, we have $A\otimes K\otimes B\in\cP$ for all $K\in\cK$ implies $A$ or $B$ in $\cP$}\] where $(\cK, \otimes, \mathbf{1})$ is an essentially small monoidal triangulated category. Denote by $Esp_{nc}(\cK), supp_{nc}$, $U_{nc}$ the relevant notations for the collection $C_{nc}=\{\text{nc-primes}\}$. By Lemma 4.1.2 \cite{nakano2022noncommutative}, for any $A,B\in\cK$, we have \[\bigcup_{K\in\cK} supp_{nc}(A\otimes K\otimes B)=supp_{nc}(A)\cap supp_{nc}(B)\] hence \[U_{nc}(A)\cup U_{nc}(B)=\bigcap_{K\in\cK} U_{nc}(A\otimes K\otimes B)\subseteq U_{nc}(A\otimes\mathbf{1}\otimes B)=U_{nc}(A\otimes B)\] This gives

\begin{lem} \label{2sided} Let $\cK$ be a monoidal triangulated category and $\cU$ an ultrafilter on $Esp_{nc}(\cK)$. Then $\cP_{\cU}$ is a proper thick two-sided ideal. \end{lem}

\begin{proof} First, for the same reason as Example \ref{ex}, $\mathbf{1}\notin\cP_{\cU}$ so it's proper. To show $\cP_{\cU}$ is a two-sided ideal, let $E\in\cP_{\cU}$ and $K\in\cK$, we have \[U_{nc}(E)\subseteq U_{nc}(E)\cup U_{nc}(K)\subseteq U_{nc}(E\otimes K)\] Similarly $U_{nc}(E)\subseteq U_{nc}(K\otimes E)$, so together we get $E\otimes K$ and $K\otimes E$ in $\cP_{\cU}$ because $U_{nc}(E)\in\cU$. \end{proof}

The containment $supp_{nc}(A\otimes B)\subseteq supp_{nc}(A)\cap supp_{nc}(B)$ is an equality for some $\cK$ and tensor product. This is called the \textbf{tensor product property}. Below we show, under the tensor product property, the noncommutative Balmer spectrum is spectral.

\begin{thm} \label{ncspectral} Let $\cK$ be a monoidal triangulated category such that the support $supp_{nc}$ possesses the tensor product property, i.e., $supp_{nc}(A\otimes B)=supp_{nc}(A)\cap supp_{nc}(B)$. Then the noncommutative Balmer spectrum of $\cK$ is spectral. \end{thm}

\begin{proof} Let $\cU$ be an ultrafilter on $Esp_{nc}(\cK)$ and Lemma \ref{2sided} tells us that $\cP_{\cU}$ is a proper thick two-sided ideal. Suppose $A\otimes K\otimes B\in\cP_{\cU}$ for all $K\in\cK$, then certainly $A\otimes B=A\otimes\mathbf{1}\otimes B\in\cP_{\cU}$ so $U_{nc}(A\otimes B)\in\cU$. A similar argument as the last argument in Example \ref{ex} will give $A$ or $B$ in $\cP_{\cU}$. \end{proof}

Please see \cite{nakano2022tensor} and references therein for a comprehensive discussion of the tensor product property and monoidal triangulated categories that satisfy this property.


\vp{0.3}

\section{Immersion}

Let $(X,\sigma)$ be a support on $\cE$, there are natural inclusion preserving maps \[as(\cE)\overset{f_{\sigma}}{\underset{g_{\sigma}} \sdarrow} subset(X)\] where $f_{\sigma}(\cN):=\bigcup\limits_{N\in\cN}\sigma(N)$ for $\cN\in as(\cE)$ and $g_{\sigma}(S):=\{E\in\cE: \sigma(E)\subseteq S\}$ for $S\in subset(X)$. Note $g_{\sigma}(S)\in as(\cE)$ by the Definition of support. Note by definition we always have $\cN\subseteq g_{\sigma}f_{\sigma}(\cN)$ and $f_{\sigma}g_{\sigma}(S)\subseteq S$.

For any $x\in X$, we can define a specialization-closed subset $W_x:=\{x'\in X: x\notin\overline{\{x'\}}\}$, consider the following subcategories \[\cP_{X,\sigma}(x):=\{E\in\cE: x\notin\sigma(E)\}\] \[\cS_{X,\sigma}(x):=g_{\sigma}(W_x)=\{E\in\cE:\sigma(E)\subseteq W_x\}\] \[\cN_{X,\sigma}(x):=g_{\sigma}(X\setminus\overline{\{x\}}) =\{E\in\cE: \sigma(E)\subseteq X\setminus\overline{\{x\}}\}\]  one checks directly that these subcategories all belong to $as(\cE)$. If we take $(X,\sigma)$ to be $(Esp_C(\cE),supp_C)$ for a subset $C\subseteq as(\cE)$, the subcategory $\cP_{X,\sigma}(x)=\cP_{Esp_C(\cE),supp_C}(x)=\cP(\overline{\{x\}})$, here the latter is the subcategory considered in subsection \ref{scri} characterizing the soberness of the space $Esp_C(\cE)$. Indeed, \[\cP(\overline{\{x\}})=\{E\in\cE: U_C(E)\cap\overline{\{x\}}\neq\emptyset\} = \{E\in\cE: x\in U_C(E)\}\] by Lemma \ref{top}

\begin{lem} \label{pre1} \hp{1}

\begin{enumerate}

\item $\cP_{X,\sigma}(x)=\cS_{X,\sigma}(x)$, if $(X,\sigma)$ is a Noetherian sober closed support.
\item $\cP_{X,\sigma}(x)=\cN_{X,\sigma}(x)$, if $(X,\sigma)$ is an open support.

\end{enumerate}

\end{lem}

\begin{proof} For the first statement, let $E\in\cE$ be an object, we have $\sigma(E)=\overline{\{x_1\}}\cup\cdots\cup\overline {\{x_r\}}$ for some $x_1,\cdots,x_r\in X$, as $\sigma(E)$ is closed and $X$ is Noetherian and sober. Therefore, \[\sigma(E)\subseteq W_x\Leftrightarrow x_i\in W_x \text{ for all } i \Leftrightarrow x\notin\overline{\{x_i\}} \text{ for all } i\Leftrightarrow x\notin\sigma(E)\] that is, $\cP_{X,\sigma}(x)=\cS_{X,\sigma}(x)$. Next, an object $E\in\cN_{X,\sigma}(x)$ if and only if $\sigma(E)\subseteq X\setminus\overline{\{x\}}$ if and only if $\sigma(E)\cap\overline{\{x\}}$ and this means $x\notin\sigma(E)$ by Lemma \ref{top} because $\sigma(E)$ is open. \end{proof}

\begin{lem} \label{pre2} \hp{1}

\begin{enumerate} \item Suppose $X$ is a $T_0$ space, then $W_x=W_y$ if and only if $x=y$. \item Let $(X,\sigma)$ be a Noetherian sober closed support and $f_{\sigma}\circ g_{\sigma}=id_{Spcl(X)}$, where $Spcl(X)=\{W\subseteq X: W \text{ is specialization-closed}\}$. Then $x\in\overline{\{y\}}$ if and only if $\cP_{X,\sigma}(x)\subseteq\cP_{X,\sigma}(y)$. \end{enumerate} \end{lem}

\begin{proof} \begin{enumerate} \item It's obvious that $x\notin W_x$ for any $x\in X$. Therefore $y\notin W_y=W_x$ give $x\in\overline{\{y\}}$, so $\overline{\{x\}}\subseteq\overline{\{y\}}$. Similarly, $x\notin W_x=W_y$ gives $\overline{\{y\}}\subseteq\overline{\{x\}}$. To sum up, $W_x=W_y$ means $\overline{\{x\}}=\overline{\{y\}}$ so it holds if and only if $x=y$, as $X$ is $T_0$. \item By Lemma \ref{pre1} (1) we have $\cP_{X,\sigma}(x)=\cS_{X,\sigma}(x)$. Also, $W_x\in Spcl(X)$ so \[x\in\overline{\{y\}}\Leftrightarrow y\notin W_x=f_{\sigma}(g_{\sigma}(W_x))=f_\sigma(\cS_{X,\sigma}(x))=f_\sigma(\cP_{X,\sigma}(x))\] \[\Leftrightarrow y\notin\sigma(E) \text{ for all } E\in\cP_{X,\sigma}(x) \] \[\Leftrightarrow E\in\cP_{X,\sigma}(y) \text{ for all } E\in\cP_{X,\sigma}(x) \] \[\Leftrightarrow \cP_{X,\sigma}(x)\subseteq\cP_{X,\sigma}(y) \qedhere \] \end{enumerate} \end{proof}

\subsection{Immersion for closed support} \label{close} Let $(X,\sigma)$ be a closed support on $\cE$, the natural inclusion preserving maps $f_{\sigma}, g_{\sigma}$ are in fact between \[as(\cE)\overset{f_{\sigma}} {\underset{g_{\sigma}} \sdarrow} Spcl(X)\]

\begin{lem} \label{tech} For any $x\in X$, we have $\overline{\{x\}}=f_{\sigma}g_{\sigma}(\overline{\{x\}})$ if and only if $\overline{\{x\}}=\sigma(E)$ for some $E\in\cE$. \end{lem}

\begin{proof} Suppose $x\in\overline{\{x\}}\subseteq f_{\sigma}g_{\sigma}(\overline{\{x\}})=\bigcup_{M\in g_{\sigma}(\overline{\{x\}})} \sigma(M)$, then there is some $E\in g_{\sigma}(\overline{\{x\}})$, i.e., $\sigma(E)\subseteq\overline{\{x\}}$ such that $x\in\sigma(E)$ and this forces $\sigma(E)=\overline{\{x\}}$ (because $\sigma(E)$ is closed). Conversely, $E\in g_{\sigma}(\overline{\{x\}})$ if $\overline{\{x\}}=\sigma(E)$ for some $E$, so we get directly $\overline{\{x\}}\subseteq f_{\sigma}g_{\sigma}(\overline{\{x\}})$ (the other containment follows from the definitions of $f_{\sigma}$ and $g_{\sigma}$). \end{proof}

Define $C_{X,\sigma}:=\{\cP_{X,\sigma}(x): x\in X\}\subseteq as(\cE)$. By Lemma \ref{pre1} $C_{X,\sigma}=\{\cS_{X,\sigma}(x): x\in X\}=\{g_{\sigma}(W_x): x\in X\}$ as $(X,\sigma)$ is closed. In this situation $g_{\sigma}$ maps $pm(X):=\{W_x\}_{x\in X}$ to $C_{X,\sigma}$. On the other direction,

\begin{prop} The map $f_{\sigma}: as(\cE)\xra{} Spcl(X)$ restricts to $C_{X,\sigma}\xra{} pm(X)$, that is, $f_{\sigma}(g_{\sigma}(W_x))\in pm(X)$ for any $x\in X$, if the equivalent conditions in Lemma \ref{tech} holds. \end{prop}

\begin{proof} In fact, we will show $f_{\sigma}(g_{\sigma}(W_x))=W_x$, for any $x\in X$. First of all, by construction we have $f_{\sigma}(g_{\sigma}(W_x))\subseteq W_x$. For the other direction, choose any $y\in W_x$, we have $x\notin\overline{ \{y\} }=\sigma(E)$ for some $E$, by Lemma \ref{tech}. Hence $\sigma(E)=\overline{\{y\}}\subseteq W_x$ as $W_x$ is specialization-closed and this gives $E\in g_{\sigma}(W_x)$. Therefore $y\in\overline{\{y\}}=\sigma(E)\subseteq f_{\sigma}g_{\sigma}(W_x)$. \end{proof}

In fact, the above shows, under the equivalent conditions in Lemma \ref{tech}, the maps $f_{\sigma}: as(\cE) \rightleftarrows Spcl(X): g_{\sigma}$ restrict to bijections $C_{X,\sigma} \rightleftarrows pm(X)$. Hence, define

\begin{defn} \label{classifying} A closed support $(X,\sigma)$ on $\cE$ is \textbf{classifying} if \begin{enumerate} \item $X$ is Noetherian and sober; \item The condition in Lemma \ref{tech} holds, that is, $\overline{\{x\}}=f_{\sigma}g_{\sigma}(\overline{\{x\}})$ for any $x\in X$. Equivalently, for any $x\in X$ we have $\overline{\{x\}}=\sigma(E)$ for some $E\in\cE$. \end{enumerate} \end{defn}

Clearly, a classifying support data in the sense of \cite{balmer2005spectrum} (resp. \cite{matsui2019singular} and \cite{matsui2021prime}) is a classifying support in the above sense if the underlying category $\cE$ is in fact a tensor triangulated (resp. triangulated) category.

\begin{lem} \label{supported} Let $(X,\sigma)$ be a classifying closed support on $\cE$. Then \begin{enumerate} \item The natural inclusion preserving maps $f_{\sigma}: Th(\cE) \rightleftarrows Spcl(X): g_{\sigma}$ restrict to inclusion preserving bijections $f_{\sigma}: C_{X,\sigma} \rightleftarrows pm(X): g_{\sigma}$. \item $\{\sigma(E)\}_{E\in\cE}$ form a closed basis for $X$. \item For any closed subset $Z$ of $X$, there is some object $E\in\cE$ such that $Z=\sigma(E)$. \end{enumerate} \end{lem}

\begin{proof} \begin{enumerate} \item This comes from the discussion right before Definition \ref{classifying}. \item This is an easy consequence of the fact that any $\overline{\{x\}}$ can be written as $\sigma(E)$ for some $E$. \item For any closed subset $Z$, we have $Z=\overline{\{x_1\}}\cup\cdots\cup\overline{\{x_r\}}$ for some $x_i\in X$ and $r\in\mathbb{N}$ because $X$ is Noetherian and sober. There are objects $E_i$ such that $\overline{\{x_i\}}=\sigma(E_i)$ and this gives $Z=\sigma(E_1)\cup\cdots\cup\sigma(E_r)=\sigma(E_1\oplus\cdots\oplus E_r)$. \qedhere \end{enumerate} \end{proof}

\begin{defn} \label{compatible} Let $C\subseteq as(\cE)$ be a subset, a support $(X,\sigma)$ is \textbf{$C$-compatible} if $\cP_{X,\sigma}(x)\in C$ for any $x\in X$. \end{defn}

\begin{lem} \label{comlem} Given any $C\subseteq as(\cE)$, the support $(Esp_C(\cE), supp_C)$ is $C$-compatible. \end{lem}

\begin{proof} For any $\cP\in Esp_C(\cE)$, $\cP_{Esp_C(\cE), supp_C}(\cP)=\{M\in\cE: \cP\notin supp_C(M)\}=\{M\in\cE: M\in\cP\}=\cP\in C$. \end{proof}

\begin{lem} \label{r-supp} Let $(X,\sigma)$ be a support and $U\subseteq X$ a subset, give $U$ the induced topological space structure and define $\sigma_U(E):=\sigma(E)\cap U$ for any $E\in\cE$. Then $(U,\sigma_U)$ is also a support for $\cE$. Moreover, \begin{enumerate} \item $(U,\sigma_U)$ is open/closed if $(X,\sigma)$ is open/closed. \item  $(U,\sigma_U)$ is $C$-compatible if $(X,\sigma)$ is $C$-compatible, for any subset $C\subseteq as(\cE)$. \item $\{\sigma_U(M)\}_{M\in\cE}$ is an open/closed basis for $U$ if $\{\sigma(M)\}_{M\in\cE}$ is an open/closed basis for $X$. \end{enumerate} \end{lem}

\begin{proof} Straightforward and $\cP_{X,\sigma}(x)=\cP_{U,\sigma_U}(x)$ for $x\in U$. \end{proof}

Let $C\subseteq as(\cE)$ be a subset and $(X,\sigma)$ a $C$-compatible support, one can define a map \[j: X\xra{} Esp_C(\cE)\] \[\hp{0.15}x\mapsto\cP_{X,\sigma}(x)\] we have

\begin{thm} \label{immersion} Let $C\subseteq as(\cE)$ be a subset, $(X,\sigma)$ a $C$-compatible closed support and $j$ be the map defined above. Then \begin{enumerate} \item $j$ is a morphism of supports, that is, $(Esp_C(\cE), supp_C)$ is the final support among all $C$-compatible closed supports. \item $j$ is an immersion if $(X,\sigma)$ is Noetherian, sober, $f_{\sigma}\circ g_{\sigma} = 1_{pm(X)}$, and $\{\sigma(E)\}_{E\in\cE}$ is a closed basis for $X$. Moreover, $j$ is a homeomorphism if and only if $j$ is surjective, that is, $C_{X,\sigma}=C$. \end{enumerate} \end{thm}

\begin{proof} Note that $j^{-1}(supp_C(E))=j^{-1}(\{\cP\in Esp_C(\cE): E\notin\cP\})=\{x\in X: E\notin\cP_{X,\sigma}(x)\}=\sigma(E)$, so $j$ is continuous as $(X,\sigma)$ is closed, and therefore it defines a morphism of supports. For the second claim, first, we need to show $j$ is injective, suppose $\cP_{X,\sigma}(x)=\cP_{X,\sigma}(y)$ for $x,y\in X$, then by Lemma \ref{pre1} this means $\cS_{X,\sigma}(x)=\cS_{X,\sigma}(y)$, i.e., $g_{\sigma}(W_x)=g_{\sigma}(W_y)$. Hence $W_x=W_y$ as $f_{\sigma}\circ g_{\sigma}=id_{pm(X)}$ and this in turn gives $x=y$ by Lemma \ref{pre2} as $X$ is $T_0$. Next, $j$ is an immersion as $\{supp_C(E)\}_{E\in\cE}$ and $\{\sigma(E)\}_{E\in\cE}$ are closed bases for $Esp_C(\cE)$ (by Lemma \ref{blem}) and $X$, respectively (so the topology on $X$ is the one induced from $Esp_C(\cE)$). \end{proof}

\begin{cor} \label{corimer} Let $C\subseteq as(\cE)$ be a subset, $(X,\sigma)$ a $C$-compatible closed support. Then $j$ is an immersion if $(X,\sigma)$ is classifying. \end{cor}

\begin{proof} By Lemma \ref{supported}, classifying support satisfies all the hypotheses of Theorem \ref{immersion} (2). \end{proof}

Any closed support $(X,\sigma)$ is certainly $C_{X,\sigma}$-compatible, so

\begin{cor} If $(X,\sigma)$ is a classifying closed support. Then there is a homeomorphism $X\cong Esp_{X,\sigma}(\cE):= Esp_{C_{X,\sigma}}(\cE)$, which restricts to a homeomorphism $\sigma(E)\cong supp_{X,\sigma}(E):=supp_{C_{X,\sigma}}(E)$. \end{cor}

\begin{cor} \label{uimmer} Let $C\subseteq as(\cE)$ be a subset, $(X,\sigma)$ a $C$-compatible closed support and $U\subseteq X$ is a subset. By Lemma \ref{r-supp}, $(U,\sigma_U)$ is also a $C$-compatible closed support on $\cE$, so there is a map $j_U: U\xra{}Esp_C(\cE)$ which sends $x\in U$ to $\cP_{U,\sigma_U}(x)=\cP_{X,\sigma}(x)$. This is an immersion if $X$ is classifying. \end{cor}

\begin{proof} It is obvious that the above map is the same as the restriction $j|_U: U\xra{}X\xra{}Esp_C(\cE)$, so it's the composition of two immersions under the hypotheses. \end{proof}

\subsubsection{Matsui's result} Let $\cT$ be a triangulated category, recall

\begin{defn} (Definition 2.8 \cite{matsui2023triangular})
  A proper thick subcategory $\cP\subsetneq\cT$ is a \textbf{prime thick subcategory} if the partially ordered set $\{\cN\in Th(\cT): \cP\subsetneq\cN\}$ has the smallest element. For simplicity, we will call these the \textbf{Matsui primes}. Choose $C$ to be the collection of such subcategories and denote it by $C_{Matsui}$. Also, denote the topological space $Esp_C(\cT)$ by $MSpc(\cT)$ and this is the so-called \textbf{Matsui spectrum} of $\cT$.
\end{defn}

For a Noetherian scheme $X$ and $Perf(X)$ the category of perfect complexes. The cohomological support $(X, supph_X)$ gives a classifying Noetherian spectral closed support for $Perf(X)$, so there is an immersion $X\hookrightarrow MSpc(Perf(X))$ if $(X, supph_X)$ is $C_{Matsui}$-compatible and this is one of the main result of Matsui (Theorem 3.4 \cite{matsui2021prime}).

By Theorem 1.4 \cite{matsui2021prime}, every prime tt-ideal is a Matsui prime for $Perf(X)$ so there is an embedding $BSpc(Perf(X))\hookrightarrow MSpc(Perf(X))$. In general, one does not know about the relation between the prime tt-ideals and Matsui primes as an arbitrary thick subcategory may or may not be a tt-ideal. For any triangulated category, denote by $thick(-)$ the smallest thick subcategory containing a collection of objects, the following can be extract from Lemma 3.6 \cite{matsui2019singular}:

\begin{lem} \label{ttt} Suppose $(\cT,\otimes,\mathbf{1})$ is a closed tensor triangulated category, that is, the functor $M\otimes-:\cT\xra{}\cT$ has a right adjoint $F(M,-):\cT\xra{}\cT$ for each $M\in\cT$, and $\cT=thick_{\cT}(\mathbf{1})$. Then \[\{\text{tt-ideals}\} = \{\text{radical tt-ideals} \} = \{\text{thick subcategories} \} \] \end{lem}

\begin{proof} It's obvious that the object $\mathbf{1}$ is dualizable (i.e., strongly dualizable, or rigid) and the subcategory of dualizable objects form a thick subcategory of $\cT$. This implies all objects in $\cT$ are dualizable because $\cT=thick_{\cT}(\mathbf{1})$. Thus, any object $M\in\cT$ belongs to the smallest tt-ideal generated by $M\otimes M$ so every tt-ideal is radical, by a Proposition 4.4 \cite{balmer2005spectrum}.

Let $\cN$ be a thick subcategory, define $I_{\cN}:=\{M\in\cT: M\otimes N\in\cN \text{ for any } N\in\cN\}$. It is straightforward that $I_{\cN}$ is a thick subcategory containing $\mathbf{1}$ so $I_{\cN}=\cT$ and this means $\cN$ is a tt-ideal. \end{proof}

Therefore

\begin{thm} Suppose $\cT$ is a closed tensor triangulated category whose Balmer spectrum $BSpc(\cT)$ is Noetherian. Then $BSpc(\cT)=MSpc(\cT)$. \end{thm}

\begin{proof} By Lemma \ref{ttt} any thick subcategory $\cN$ in $\cT$ is a radical tt-ideal. If $\cN$ is a Matsui prime, then it's a prime tt-ideal, by Proposition 4.8 \cite{matsui2021prime}. Conversely, suppose $\cN$ is a prime tt-ideal. $BSpc(\cT)$ is Noetherian so by Proposition 4.7 \cite{matsui2021prime} the set $\{\cI\in Rad_{\otimes}(\cT): \cN\subsetneq\cI \}$ has a unique minimal element (here $Rad_{\otimes}(\cT)$ stands for the set of radical tt-ideals of $\cT$). By Lemma \ref{ttt}, this set is in fact $\{\cI\in Th(\cT): \cN\subsetneq\cI \}$ so $\cN$ is a Matsui prime. \end{proof}

\begin{cor} Let $k$ be a field of characteristic $p>0$ and $G$ a finite group such that $p$ divides the order of $G$. Then $MSpc(kG$-$stab)=BSpc(kG$-$stab)$, where $kG$-$stab$ is the stable module category. \end{cor}

\begin{proof} The category $kG$-$stab$ is rigid, $kG$-$stab=thick_{kG-stab}(k)$ and $BSpc(kG$-$stab)\cong Proj\,H^*(G;k)$ is Noetherian. \end{proof}

One can deduce the above result by Theorem 2.13, Lemma 3.6, Corollary 3.7, Theorem 3.13 of \cite{matsui2019singular} and Theorem 2.16 of \cite{matsui2021prime}. We record and rephrase it in the current form here so there is a reference for future usage.

\subsection{Immersion for open support} \label{s5}

We discuss immersion for open supports in this subsection. Recall from Section \ref{sec2} that the Hochster dual $X^{\vee}$ of a spectral space $X$ is a new topology on the underlying set of $X$ by taking as basic open subsets the closed sets of $X$ with quasi-compact complements in $X$.

\begin{lem} \label{bopen} Let $C\subseteq as(\cE)$ be a subset and the space $Esp_C(\cE)$ is spectral. Then the basic open subsets of $(Esp_C(\cE))^{\vee}$ are $\bigcap\limits_{1\leqslant i\leqslant n} supp_C(M_i)$ for some objects $M_i$ and $n\in\mathbb{N}$. In particular, $supp_C(M)$ is open in $(Esp_C(\cE))^{\vee}$ for all $M\in\cE$. \end{lem}

\begin{proof} By definition, a basic open subset of $(Esp_C(\cE))^{\vee}$ is a closed subset $Z_C(\cX)$ whose complement $Esp_C(\cE)\setminus Z_C(\cX)$ is quasi-compact in $Esp_C(\cE)$, for some collection of objects $\cX$. Therefore \[Esp_C(\cE)\setminus Z_C(\cX)=Esp_C(\cE)\setminus (\bigcap_{M\in\cX} supp_C(M))=\bigcup_{M\in\cX} (Esp_C(\cE)\setminus supp_C(M))=\bigcup_{M\in\cX} U_C(M)\] This is an open cover (in $Esp_C(\cE)$) of $Esp_C(\cE)\setminus Z_C(\cX)$ so there are finitely many objects $M_i$ ($1\leqslant i\leqslant n$) such that $Esp_C(\cE)\setminus Z_C(\cX)=\bigcup\limits_{1\leqslant i\leqslant n} U_C(M_i)$. Hence $Z_C(\cX)=\bigcap\limits_{1\leqslant i\leqslant n} supp_C(M_i)$. \end{proof}

Suppose $(X,\sigma)$ is an open support, the natural inclusion preserving maps $f_{\sigma}, g_{\sigma}$ now become to maps between \[as(\cE)\overset{f_{\sigma}}{\underset{g_{\sigma}} \sdarrow} open(X)\]

\begin{lem} \label{pre3} $x\in\overline{\{y\}}$ if and only if $\cN_{X,\sigma}(x)\supseteq\cN_{X,\sigma}(y)$, if $f_{\sigma}\circ g_{\sigma}=1_{op}$. \end{lem}

\begin{proof} Indeed, \[x\in\overline{\{y\}}\Leftrightarrow \overline{\{x\}}\subseteq\overline{\{y\}}\Leftrightarrow X\setminus\overline{\{x\}}\supseteq X\setminus\overline{\{y\}}\] hence if and only if $\cN_{X,\sigma}(x)=g_{\sigma}(X\setminus\overline{\{x\}})\supseteq g_{\sigma}(X\setminus\overline{\{y\}})=\cN_{X,\sigma}(y)$. \end{proof}

\begin{lem} \label{optech} Suppose $(X,\sigma)$ is an open support for $\cE$ and $x\in X$ be any point. Then $X\setminus\overline{\{x\}}=\sigma(E)$ for some $E\in\cE$ implies $X\setminus\overline{\{x\}}\subseteq f_{\sigma}g_{\sigma}(X\setminus\overline{\{x\}})$ (so in fact $X\setminus\overline{\{x\}}=f_{\sigma}g_{\sigma}(X\setminus\overline{\{x\}})$). The converse holds if $X$ is Noetherian. \end{lem}

\begin{proof} The first statement is obvious. Conversely, suppose \[X\setminus\overline{\{x\}}=f_{\sigma}g_{\sigma}(X\setminus\overline{\{x\}}) =\bigcup\limits_{\begin{matrix} M \text{ s.t.} \\ \sigma(M)\subseteq X\setminus\overline{\{x\}} \end{matrix} } \sigma(M)\] the latter is a covering of $X\setminus\overline{\{x\}}$ so there is a finite subcover $\bigcup\limits_{i=1}^n\sigma(M_i)=\sigma(M_1\oplus\cdots M_n)$ for some $i$ as $X$ is Noetherian (therefore compact). \end{proof}

\begin{defn} \label{opclassifying} An open support $(X,\sigma)$ on $\cE$ is \textbf{classifying} if \begin{enumerate} \item $X$ is $T_0$ and $\{\sigma(E)\}_{E\in\cE}$ form an open basis for $X$; \item If, for any $x\in X$, we have $X\setminus\overline{\{x\}}=\sigma(E)$ for some $E\in\cE$. \end{enumerate} \end{defn}

For any classifying open support $(X,\sigma)$, we get $X\setminus\overline{\{x\}}=f_{\sigma}g_{\sigma} (X\setminus\overline{\{x\}})$ for any $x\in X$ by Lemma \ref{optech}, hence

\begin{lem} \label{opsupported} Let $(X,\sigma)$ be a classifying open support on $\cE$. Then the natural inclusion preserving maps $f_{\sigma}: as(\cE) \rightleftarrows open(X): g_{\sigma}$ restrict to inclusion preserving bijections $f_{\sigma}: C_{X,\sigma} \rightleftarrows \{X\setminus\overline{\{x\}}\}_{x\in X}: g_{\sigma}$. \end{lem}

\begin{proof} This is Straightforward. Recall $C_{X,\sigma}=\{\cP_{X,\sigma}(x): x\in X\}$ and by Lemma \ref{pre1} $C_{X,\sigma}=\{\cN_{X,\sigma}(x): x\in X\}=\{g_{\sigma}(X\setminus\overline{\{x\}}): x\in X\}$ as $(X,\sigma)$ is open. \end{proof}

Suppose $Esp_C(\cE)$ is a spectral space for some $C\subseteq as(\cE)$ and $(X,\sigma)$ a $C$-compatible open support for $\cE$, there is a natural map \[j^{\vee}: X\xra{} (Esp_C(\cE))^{\vee} \] \[\hp{0.58}x\mapsto\cP_{X,\sigma}(x)=\cN_{X,\sigma}(x)\] Note that $((Esp_C(\cE))^{\vee}, supp_C)$ is a spectral $C$-compatible open support (Lemma \ref{bopen}) in this situation.

\begin{thm} \label{opimmersion} Let $(X,\sigma)$ be a $C$-compatible open support for some $C\subseteq as(\cE)$. If $Esp_C(\cE)$ is spectral (so $j^{\vee}$ can be defined). Then

\begin{enumerate}

\item $j^{\vee}$ is a morphism of supports, that is, $((Esp_C(\cE))^{\vee}, supp_C)$ is the final support among all $C$-compatible open supports.

\item $j^{\vee}$ is an immersion, if $X$ is $T_0$, $f_{\sigma}\circ g_{\sigma}=id_{op}$, and $\{\sigma(E)\}_{E\in\cE}$ form an open basis for $X$. Moreover, $j^{\vee}$ is a homeomorphism if and only if $j^{\vee}$ is surjective, that is, $C_{X,\sigma}=C$.

\item In addition, if $X$ is spectral, $X^{\vee}$ is Noetherian and any open subset of $X$ has the form $\sigma(E)$ for some $E\in\cE$, then $j^{\vee}$ induces another immersion $j:X^{\vee}\xra{}Esp_C(\cE)$.

\end{enumerate}
\end{thm}

\begin{proof} Note that we have $(j^{\vee})^{-1}(supp_C(E))=(j^{\vee})^{-1}(\{\cP\in Esp_C(\cE): E\notin\cP\})=\{x\in X: E\notin\cP_{X,\sigma}(x)\}=\sigma(E)\quad (*)$.

\begin{enumerate}

\item By Lemma \ref{bopen}, take a basic open subset $\bigcap\limits_{1\leqslant i\leqslant n} supp_C(M_i)$ of $(Esp_C(\cE))^{\vee}$, then \[(j^{\vee})^{-1}(\bigcap\limits_{1\leqslant i\leqslant n} supp_C(M_i))=\{x\in X: \cP_{X,\sigma}(x)\in supp_C(M_i) \text{ for all } i\}\] \[=\bigcap\limits_{1\leqslant i\leqslant n} (j^{\vee})^{-1}(supp_C(M_i))\] \[=\bigcap\limits_{1\leqslant i\leqslant n} \sigma(M_i)\] is also open so $j^{\vee}$ is continuous, and therefore a morphism of supports (because $(*)$ also holds).

\item To show $j^{\vee}$ is injective, note that $\cN_{X,\sigma}(x)=\cN_{X,\sigma}(y)$ if and only if $\overline{\{x\}}=\overline{\{y\}}$ by Lemma \ref{pre3}, and the latter is equivalent to $x=y$ since $X$ is $T_0$. To show it's an immersion, it's enough to show $j^{\vee}(\sigma(E))=supp_C(E)\cap j^{\vee}(X)$ for any $E\in\cE$. Indeed, for any $\cP_{X,\sigma}(x)\in j^{\vee}(\sigma(E))=\{\cP_{X,\sigma}(x): x\in\sigma(E)\}$, we have $E\notin \cP_{X,\sigma}(x)$. We also have $\cP_{X,\sigma}(x)\in C$ as $(X,\sigma)$ is $C$-compatible so together gives $\cP_{X,\sigma}(x)\in supp_CE$. Conversely, let $\cP\in supp_C(E)\cap j^{\vee}(X)$, so \begin{enumerate} \item From $\cP\in j^{\vee}(X)$ we get $\cP=\cP_{X,\sigma}(x)$ for some $x\in X$, and \item From $\cP\in supp_CE$ we get $E\notin\cP\stackrel{(\text{a})}{=}\cP_{X,\sigma}(x)$, i.e., $x\in\sigma(E)$. \end{enumerate} Together gives $\cP=\cP_{X,\sigma}(x)$ from some $x\in\sigma(E)$, i.e., $\cP\in j^{\vee}(\sigma(E))$. Now, for any open subset $U=\bigcup\limits_i\sigma(E_i)$ in $X$, we have $j^{\vee}(\bigcup\limits_i\sigma(E_i))=\bigg(\bigcup\limits_i supp_C(E_i)\bigg)\cap j^{\vee}(X)$, i.e, $j^{\vee}$ is a homeomorphism onto the subspace $j^{\vee}(X)$ of $(Esp_C(\cE))^{\vee}$.

\item Note that $j$ and $j^{\vee}$ are the same map with different topologies in the source and target so in particular $j$ is also injective and \[j^{-1}(supp_C(E))=\sigma(E)\] \[j(\sigma(E))=supp_C(E)\cap j(X^{\vee})\] from previous discussion. The first equality implies $j$ is also continuous because $\{supp_C(E)\}_{E\in\cE}$ is a closed basis of $Esp_C(\cE)$ and $\sigma(E)$ is closed in $X^{\vee}$ ($X^{\vee}$ is Noetherian spectral and Lemma \ref{nsspace}). Also, from Lemma \ref{nsspace} and the assumption, any closed subset of $X^{\vee}$ has the form $\bigcap_{i\in I}\sigma(E_i)$ for some $E_i\in\cE$. Apply the second equality we get \[j(\bigcap_{i\in I}\sigma(E_i))=\bigcap_{i\in I}(supp_C(E_i)\cap j(X^{\vee}))\qquad \text{(since $j$ is injective)}\] \[\hp{0.38}=\bigg(\bigcap_{i\in I} supp_C(E_i)\bigg)\cap j(X^{\vee})=Z_C(\cS)\cap j(X^{\vee})\] where $\cS:=\{E_i: i\in I\}$, so the image of any closed subset in $X^{\vee}$ is the intersection of a closed subset in $Esp_C(\cE)$ with $j(X^{\vee})$. \qedhere

\end{enumerate}
\end{proof}

\begin{cor} \label{ac} Let $C\subseteq as(\cE)$ be a subset such that $Esp_C(\cE)$ is spectral, $(X,\sigma)$ is a $C$-compatible open support. Then $j^{\vee}$ is an immersion if $(X,\sigma)$ is classifying. \end{cor}

\begin{proof} By Definition of classifying open support and Lemma \ref{opsupported}, all hypotheses of Theorem \ref{opimmersion} (2) are satisfied. \end{proof}

Any open support $(X,\sigma)$ is certainly $C_{X,\sigma}$-compatible, so

\begin{cor} If $(X,\sigma)$ is a classifying open support and suppose $Esp_{X,\sigma}(\cE):= Esp_{C_{X,\sigma}}(\cE)$ is spectral. Then there is a homeomorphism $X\cong (Esp_{X,\sigma}(\cE))^{\vee}$, which restricts to a homeomorphism $\sigma(E)\cong supp_{X,\sigma}(E)$. \end{cor}

\begin{ex} \label{atom} If $\cE$ is in fact an abelian category, then the atom spectrum and support together give a classifying open support $(ASpec(\cE), ASupp)$. Indeed, $ASpec(\cE)$ is $T_0$ by Proposition 3.5 \cite{kanda2015specialization} and $\{ASupp(E)\}_{E\in\cE}$ is an open basis of $ASpec(\cE)$ by Proposition 3.2 \cite{kanda2015specialization}. In fact, the same Proposition states that all open subsets of $ASpec(\cE)$ have the form $ASupp(E)$ for some $E$ if $\cE$ is a Grothendieck category. Furthermore, the atom support $(ASpec(\cE), ASupp)$ is both $C_{loc}-$ and $C_{Serre}$-compatible in Grothendieck categories because atom support of (possible infinite) direct sums equals to the union of atom supports (Proposition 2.12 \cite{kanda2015specialization}), where $C_{loc}$ and $C_{Serre}$ are respectively the subsets of all localizing and Serre subcategories. \end{ex}

\subsubsection{Immersion of schemes for Grothendieck categories} Let $\cG$ be a locally Noetherian and locally finitely presented Grothendieck category (so $\cG$ is locally coherent, see for example Example 3.2 \cite{saorin2017locally}). For such categories, the class of finitely presented objects coincide with the class of Noetherian objects and we denote this by $Noeth(\cG)$. Following Section 14.1.2 \cite{Prest2009Purity}, the \textbf{Gabriel spectrum} $GSpec(\cG)$ is the topological space with underlying set the isomorphism classes of indecomposable injective objects $Inj(\cG)$, whose basis of open subsets are those of the form $\{X\in Inj(\cG): Hom_{\cG}(A,X)=0\}$ for $A\in Noeth(\cG)$. The \textbf{Ziegler spectrum} $ZSpec(\cG)$ of $\cG$, following \cite{herzog1997ziegler, krause1997the}, is the topological space with the same underlying set as the Gabriel spectrum but a different basis of open subsets: $\{X\in PInj(\cG): Hom_{\cG}(A,X)\neq 0\}$ for $A\in Noeth(\cG)$.

It is show in \cite{kanda2012classifying} that the atom spectrum $ASpec(\cG)$ is homeomorphic to the Ziegler spectrum $ZSpec(\cG)$ when $\cG$ is locally Noetherian. Also, by Theorem 14.1.6 \cite{Prest2009Purity} (or See Remark 4.16 \cite{bird2023homological}) the Gabriel spectrum is the Hochster dual of $ZSpec(\cG)$, when $\cG$ is locally coherent. To sum up, we have

\begin{lem} \label{agz} Suppose $\cG$ is a locally Noetherian and locally finitely presented Grothendieck category, then $ASpec(\cG)^{\vee}\cong ZSpec(\cG)^{\vee}\cong GSpec(\cG)$. \end{lem}

Suppose $X$ is a Noetherian scheme, then the category of quasi-coherent sheaves $QCoh(X)$ on $X$ is a locally Noetherian Grothendieck category, it's also locally finitely presented (Corollary 6.9.12 \cite{dieudonne1971elements} or Proposition 7 \cite{garkusha2010classifying}). The Gabriel spectrum of $QCoh(X)$ is homeomorphic to $X$ \cite{gabriel1962thesis}. Also, a localizing subcategory $\cL$ of a Grothendieck category $\cG$ is \textbf{prime} in the sense of Kanda (See Definition 6.3 \cite{kanda2015specialization}) if $\cG/\cL$ is \textbf{local}, that is, there is a simple object $S$ in $\cG/\cL$ such that the injective envelope of $S$ is a cogenerator of $\cG/\cL$. By Theorem 6.8 \cite{kanda2015specialization}, $C_{ASpec(\cG),ASupp}=C_{ploc}$, the collection of all prime localizing subcategories. We have

\begin{cor} Let $X$ be a Noetherian scheme, then there is an immersion \[j^{\vee}: ASpec(QCoh(X))\cong GSpec(QCoh(X))^{\vee}\cong X^{\vee} \xra{} Esp_{Serre}(QCoh(X))^{\vee}\] and this induces an immersion $j: X\xra{} Esp_{Serre}(QCoh(X))$. The image of the immersion is $C_{ploc}$, that is, $X\cong Esp_{ploc}(QCoh(X))$. \end{cor}

\begin{proof} $j^{\vee}$ is an immersion by Corollary \ref{ac} since $(ASpec(QCoh(X)), ASupp)$ is a classifying $C_{Serre}$-compatible open support (as recalled above in Example \ref{atom}) and $Esp_{Serre}(\cE)$ is spectral for any extriangulated category $\cE$ by Example \ref{serre}. The identification with Gabriel spectrum is Lemma \ref{agz}. Finally, $j$ is an immersion by Theorem \ref{opimmersion} (3). Note that all open subsets in the atom spectrum have the form $ASupp(E)$ for some $E\in\cE$ (again, this has been recalled in Example \ref{atom}). \end{proof}

A similar immersion from a Noetherian scheme $X$ to $Esp_{Serre}(Coh\,X)$, the Serre spectrum of its coherent sheaves category, is discussed in Section 5.3 \cite{saito2024spectrum} and the image is identify with the \textbf{meet-irreducible} Serre subcategories of $Coh\,X$.

\vp{0.3}

\section{Group action}

We discuss spectra for extriangulated categories under finite group actions in this section. Let $\cE$ to be an extriangulated category and $G$ be a group (written multiplicatively and with identity element $e$) acting on $\cE$, that is, we have

\begin{itemize}

\item an exact equivalence $T_g: \cE\xra{}\cE$ for each $g\in G$. (Recall an additive functor between extriangulated categories is \textbf{exact} if it takes conflations to conflations).
\item for any pair of elements $g,h\in G$, there is a natural isomorphism $\gamma_{g,h}:T_gT_h\xra{\cong} T_{gh}$ such that the diagram \[\xymatrix{ T_gT_hT_k\ar[r]^{T_g\gamma{h,k}} \ar[d]_{\gamma_{g,h}T_k} & T_gT_{hk}\ar[d]^{\gamma{g,hk}} \\ T_{gh}T_k\ar[r]_{\gamma_{gh,k}} & T_{ghk} }\] commutes for any $g,h,k\in G$.

\end{itemize}

\begin{rem} We require the equivalences $T_g:\cE\xra{}\cE$ to be exact equivalences for all $g\in G$. One can find examples of such actions for abelian categories in \cite{sun2019note} and triangulated categories in \cite{huang2023group}. \end{rem}

It follows from the definition that the equivalence $T_e$ is naturally isomorphic to the identity functor $id:\cE\xra{}\cE$. A \textbf{$G$-equivariant object} in $\cE$ is a pair $(E,\varphi)$ consisting of an object $E\in\cE$ and a family $\varphi=\{\varphi_g\}_{g\in G}$ of isomorphisms $\varphi_g: T_gE\cong E$ such that the diagram \[\xymatrix{T_gT_hE\ar[r]^{\gamma_{g,h}^E}\ar[d]_{T_g\varphi_h} & T_{gh}E\ar[d]^{\varphi_{gh}} \\ T_gE\ar[r]_{\varphi_g} & E }\] commutes for any $g,h\in G$. The \textbf{equivariantization} $\cE^G$ of $\cE$ is the category with objects the $G$-equivariant objects and morphisms $(E,\varphi_E)\xra{}(F,\varphi_F)$ are morphisms $E\xra{}F$ in $\cE$ that commute with the action.

\begin{lem} (Lemma 4.6 (2) (4) \cite{drinfeld2010braided}) \label{drinfeld} Suppose $G$ is a finite group acting on an idempotent complete $\mathbb{K}$-linear category $\cC$ with $\mathbb{K}$ an algebraically closed field of characteristic zero. Then there are two natural adjoint functors $Ind\tl\mathscr{F}\tl Ind$ between $\cC$ and $\cC^G$: \[\xymatrix{\cC\ar@/^/[rr]^{Ind} && \cC^G \ar@/^/[ll]^{\mathscr{F}} \\ A\ar@{|->}[rr] && (\bigoplus_g T_gA, \tau^A) \\ A && (A,\varphi)\ar@{|->}[ll] }\] where $\tau^A=\{\tau_h^A\}_{h\in G}$ is the evident equivariant structure given by $\tau_h^A=\bigoplus_g \gamma_{h,h^{-1}g}^A: T_h(\bigoplus_g T_gA)\xra{\cong} \bigoplus_g T_gA$, and each object $(A,\varphi)\in\cC^G$ is a direct summand of $Ind(A)$. \end{lem}

For any extriangulated category $\cE$ with a finite $G$-action, the equivariantization $\cE^G$ may or may not has an induced extriangulated category structure. Also, as our following results depend on Lemma \ref{drinfeld}, we would love $\cE$ to be idempotent complete. We would love $\cE$ to be idempotent complete and its equivariantization $\cE^G$ carries a naturally induced extriangulated category structure.

\begin{ex} Examples of extriangulated categories with the above properties include

\begin{enumerate}

\item (Section 4 \cite{chen2015monadicity}) All abelian categories. Indeed, the equivariantization category is abelian, where a sequence is exact if and only if the corresponding sequence of the underlying objects in the original abelian category is exact. In particular, the two natural functors $Ind$ and $\mathscr{F}$ in Lemma \ref{drinfeld} are both exact.

\item (Section 6 \cite{huang2023group}) If a triangulated category $\cT$ is $R$-linear for some ring $R$ where $|G|\neq 0$ in $R$, and $\cT$ has a DG enhancement, then $\cT^G$ has a natural triangulated category structure whose triangles are distinguished if and only if it's distinguished under the forgetful functor $\mathscr{F}$ described above. Note that this also forces $\cT$ to be idempotent complete (see for example Section 3 \cite{canonaco2022uniqueness}) and the two natural functors $Ind$ and $\mathscr{F}$ in Lemma \ref{drinfeld} are both triangulated.

\end{enumerate}
\end{ex}

From now on, we fix $\cE$ to be an idempotent complete extriangulated category such that \begin{itemize} \item the equivariantization $\cE^G$ carries a naturally induced extriangulated category structure; \item Lemma \ref{drinfeld} holds for $\cE$ and the two functors $Ind$, $\mathscr{F}$ are exact. \end{itemize} although some of the following results hold for more general extriangulated categories. In particular, if $X$ is a scheme of finite type over a field, the abelian category of coherent sheaves $Coh(X)$ (Example 2.13 \cite{sun2019note}) and triangulated category of perfect complexes $Perf(X)$ (Example 6.4 \cite{huang2023group}) are examples of such categories.

It is straightforward that if $\cN\in as(\cE)$ (resp. $Th(\cE)$), then $T_g(\cN):=\{T_g N: N\in\cN \}$ for any $g\in G$ is also a subcategory in $as(\cE)$ (resp. $Th(\cE)$), because $T_g$ is exact and $\gamma_{g,g^{-1}}: T_gT_{g^{-1}}\xra{\cong}T_e$ is naturally isomorphic to the identity functor as mentioned before.

\begin{defn} A subcategory $\cN$ is \textbf{$G$-invariant} if $T_g(\cN)\subseteq\cN$ for all $g\in G$. Denote the collection of $G$-invariant subcategories in $as(\cE)$ and $Th(\cE)$ respectively by $as_G(\cE)$ and $Th_G(\cE)$. \end{defn}

If $\cN$ is $G$-invariant, then we have $\cN=T_e(\cN)=(T_gT_{g^{-1}})(\cN)=T_g(T_{g^{-1}}(\cN))\subseteq T_g(\cN)\subseteq\cN$ so in fact $T_g(\cN)=\cN$. For subcategories $\cI\in as(\cE)$ and $\cJ\in as(\cE^G)$, define \[\cI^G:=\{(E,\varphi)\in\cE^G: E\in\cI\}\] \[\cJ\cap\cE:=\{E\in\cE: Ind(E)\in\cJ\} \] It is obvious that the above functors $(-)^G$ and $-\cap\cE$ are order-preserving.

\begin{lem} \label{a1} Suppose $\cI\in as(\cE)$ and $\cJ\in as(\cE^G)$, then

\begin{enumerate}

\item $\cI^G\in as(\cE^G)$ and $\cJ\cap\cE\in as(\cE)$. In addition, $\cI^G\in Th(\cE^G)$ if $\cI\in Th(\cE)$, and $\cJ\cap\cE\in Th(\cE)$ if $\cJ\in Th(\cE^G)$.
\item $\cI^G\cap\cE\subseteq\cI$ and $(\cJ\cap\cE)^G\subseteq\cJ$.

\end{enumerate}
\end{lem}

\begin{proof} Part $(1)$ is straightforward since $\mathscr{F}$ and $Ind$ are exact functors. For (2), we have \[\cI^G\cap\cE=\{(E,\varphi)\in\cE^G: E\in\cI\}\cap\cE=\{F\in\cE: \mathscr{F}(Ind(F))\in\cI\}\] For any $F\in \cI^G\cap\cT$, that is, $\oplus_g T_gF\in\cI$, then $T_gF\in\cI$ because $\cI$ is closed under direct summands. Choose $g=e$ we get $F=T_eF\in\cI$. Now, consider \[(\cJ\cap\cE)^G=\{(F,\varphi)\in\cE^G: F\in\{E\in\cT: Ind(E)\in\cJ\}\}=\{(E,\varphi)\in\cE^G: Ind(E)\in\cJ\}\] For any $(E,\varphi)\in (\cJ\cap\cE)^G$, it's an object in $\cE^G$, hence a direct summand of $Ind(E)\in\cJ$ by Lemma \ref{drinfeld}, so $(E,\varphi)\in\cJ$ because $\cJ$ is closed under direct summands. \end{proof}

\begin{defn} A subcategory $\cJ\in as(\cE^G)$ is \textbf{closed under Ind} or \textbf{Ind-closed} if $Ind(E)\in\cJ$ for any $(E,\varphi)\in\cJ$. Denote the collection of such subcategories by $as_{Ind}(\cE^G)$. Similarly, we use $Th_{Ind}(\cE^G)$ to denote the collection of Ind-closed thick subcategories. \end{defn}

\begin{prop} \label{a2} Let $G$ be a finite group acting on $\cE$. Suppose $\cI\in as_G(\cE)$ and $\cJ\in as_{Ind}(\cE^G)$. Then \begin{enumerate} \item $\cI^G\in as_{Ind}(\cE^G)$ and $\cJ\cap\cE\in as_G(\cE)$. In addition, $\cI^G\in Th_{Ind}(\cE^G)$ if $\cI\in Th_G(\cE)$, and $\cJ\cap\cE\in Th_G(\cE)$ if $\cJ\in Th_{Ind}(\cE^G)$. \item $\cI^G\cap\cE=\cI$ and $(\cJ\cap\cE)^G=\cJ$. \end{enumerate} \end{prop}

\begin{proof} First of all, $\cI^G$ is $Ind$-closed because $(E,\varphi)\in\cI^G$ means $E\in\cI$, so $T_gE\in\cI$ because $\cI$ is $G$-invariant, and therefore $\bigoplus_g T_gE\in\cI$ since $\cI$ is closed under finite direct sums and $G$ is finite. Next, to show $\cJ\cap\cE$ is $G$-invariant, it's enough to show $E\in\cJ\cap\cE$ implies $T_gE\in\cJ\cap\cE$ for any $g\in G$, i.e., $Ind(E)\in\cJ$ implies $Ind(T_gE)\in\cJ$. This is obvious because $Ind(E)\cong Ind(T_gE)$ by construction.

For (2), we just need to show $\cI^G\cap\cE\supseteq\cI$ and $(\cJ\cap\cE)^G\supseteq\cJ$ because the remaining containments are given by Lemma \ref{a1}. Recall, as discussed in the proof of Lemma \ref{a1}, we have \[\cI^G\cap\cE=\{F\in\cE: \mathscr{F}(Ind(F))\in\cI\}\] \[(\cJ\cap\cE)^G=\{(E,\varphi)\in\cE^G: Ind(E)\in\cJ\}\] Pick any $F\in\cI$, we have $T_gF\in\cI$ as $\cI$ is $G$-invariant, hence $\mathscr{F}(Ind(F))=\oplus_g T_gF$ also belongs to $\cI$ because $\cI$ is closed under finite direct sums and $G$ is finite. Finally, take any $(E,\varphi)\in\cJ$, we have $Ind(E)\in\cJ$ because $\cJ$ is assumed to be $Ind$-closed. \end{proof}

To sum up, we have two commutative diagrams: \[\xymatrixcolsep{1pc}\xymatrix{ Th(\cE)\ar@/^/[rr]^{(-)^G} && Th(\cE^G) \ar@/^/[ll]^{-\cap\cE} \\  \\ Th_G(\cE)\ar@/^/[rr]^{(-)^G}\ar@{^{(}->}[uu] && Th_{Ind}(\cE^G)\ar@{^{(}->}[uu] \ar@/^/[ll]^{-\cap\cE} } \hp{0.5} \xymatrix{ as(\cE)\ar@/^/[rr]^{(-)^G} && as(\cE^G) \ar@/^/[ll]^{-\cap\cE} \\  \\ as_G(\cE)\ar@/^/[rr]^{(-)^G}\ar@{^{(}->}[uu] && as_{Ind}(\cE^G)\ar@{^{(}->}[uu] \ar@/^/[ll]^{-\cap\cE} }\] where the diagram on the left sits inside the one on the right, and the bottom maps give an order-preserving bijection between $Th_G(\cE)$, $Th_{Ind}(\cE^G)$ and $as_G(\cE)$, $as_{Ind}(\cE^G)$ respectively. Furthermore, we have the following characterization:

\begin{lem} \label{a3} Suppose $\cI\in as(\cE)$, then it is $G$-invariant if and only if $\cI^G\cap\cE=\cI$. Similarly, for any $\cJ\in as(\cE^G)$, it's $Ind$-closed if and only if $(\cJ\cap\cE)^G=\cJ$. \end{lem}

\begin{proof}

Proposition \ref{a2} means $\cI$ is $G$-invariant implies $\cI^G\cap\cE=\cI$ and $\cJ$ is $Ind$-closed implies $(\cJ\cap\cE)^G=\cJ$ so we just need to prove the converses. Suppose $\cI^G\cap\cE=\cI$, for any $F\in\cI$, we have $F\in\cI^G\cap\cE$ so $\bigoplus_g T_gF\in\cI$, hence $T_gF\in\cI$ for any $g$ as $\cI$ is closed under direct summands. This means $T_g\cI\subseteq\cI$ so $\cI$ is $G$-invariant. Next, for any $(E,\varphi)\in\cJ$, by $\cJ=(\cJ\cap\cE)^G$ we get $Ind(E)\in\cJ$ so it's $Ind$-closed. \end{proof}

Hence, for any $\cJ\in as_{Ind}(\cE^G)$, we have $\cJ=(\cJ\cap\cE)^G$ and $\cJ\cap\cE\in as_G(\cE)$, that is

\begin{cor} Any element in $as_{Ind}(\cE^G)$ has the form $\cI^G$ for some $\cI\in as_G(\cE)$. \end{cor}

For any subset $C\subseteq as_G(\cE)\subseteq as(\cE)$, define $C^G:=\{\cI^G: \cI\in C\}$ and we have $C^G\subseteq as_{Ind}(\cE^G)\subseteq as(\cE^G)$,

\begin{thm} \label{actionthm} Let $C$ be a subset of $as_G(\cE)$. Then

\begin{enumerate}

\item Let $E$ be an object in $\cE$, then $\cI\in supp_C(E)$ if and only if $\cI^G\in supp_{C^G}(Ind(E))$.
\item Let $(E,\varphi)$ be an object in $\cE^G$, then $\cI^G\in supp_{C^G}(E,\varphi)$ if and only if $\cI\in supp_C(E)$.

\item Let $D$ be a subset of $as(\cE^G)$ containing $C^G$. Then there is an immersion \[Esp_C(\cE)\cong Esp_{C^G}(\cE^G)\hookrightarrow Esp_D(\cE^G)\] This is a homeomorphism if and only if $D=C^G$.

\end{enumerate}
\end{thm}

\begin{proof}

\begin{enumerate}

\item $\cI\in supp_C(E)$ means $E\notin\cI$, then $Ind(E)\notin\cI^G$ so $\cI^G\in supp_{C^G}(Ind(E))$. If not, $Ind(E)=(\bigoplus_g T_gE,\tau^E)\in\cI^G$ means $\bigoplus_g T_gE\in\cI$ and this implies $E\cong T_eE\in\cI$, a contradiction. Conversely, if we have $E\in\cI=\cI^G\cap\cE$, by definition this means $Ind(E)\in\cI^G$, this is also a contradiction.

\item Suppose $\cI\in supp_C(E)$, so $E\notin\cI$ and this gives $(E,\varphi)\notin\cI^G$ since $(E,\varphi)\in\cI^G$ by definition means $E\in\cI$. Conversely, $(E,\varphi)\notin\cI^G$ gives $E\notin\cI$. If not, $E\in\cI$ gives $(E,\varphi)\in\cI^G$.

\item For the homeomorphism, it's enough to show the two maps of the bijection in Lemma \ref{a2} are continuous and this is (1) and (2) above. The immersion is given by Lemma \ref{two}. \qedhere

\end{enumerate}
\end{proof}

Suppose $\cE=\cK$ is in fact a monoidal triangulated category in the sense of \cite{huang2023group} and the $G$-action on it is monoidal (i.e., the exact equivalences $T_g$ also preserve the monoidal product). Choose $C$ to be the subset of $G$-invariant nc-primes in $\cK$, then Proposition 6.5 \cite{huang2023group} tells us that $C^G$ is the subset of all nc-primes in $\cK^G$, hence there is a homeomorphism $Esp_{G-nc}(\cK)\cong Esp_{nc}(\cK^G)$. This shows the above Theorem is a generalization of the result in \cite{huang2023group} and of course reduces to

\begin{prop} \label{equibalmer} Let $\cT$ be a tensor triangulated category equipped with a monoidal group action by a finite group $G$ such that $\cT^G$ has the canonical triangulated category structure. Then there is a homeomorphism $BSpc(\cT^G)\cong BSpc_G(\cT)$. \end{prop}

The key to prove Proposition 6.5 \cite{huang2023group} is:

\begin{lem} (Lemma 6.1 \cite{huang2023group}) \label{lem61} Let $\cC$ be a monoidal category equipped with a monoidal group action by a finite group $G$. Then $Ind(A\otimes\mathscr{F}(B,\varphi))\cong Ind(A)\otimes (B,\varphi)$, for any object $A\in\cC$ and $(B,\varphi)\in\cC^G$. \end{lem}

We record it here for later usage.

\subsection{Group action and Matsui spectrum for triangulated categories} Let $\cE=\cT$ be a triangulated category and $\cP$ a Matsui prime thick subcategory, that is, $\cP\subsetneq\cT$ is a thick subcategory such that the inclusion poset $\{\cN\in Th(\cT): \cP\subsetneq\cN\}$ has the smallest element. Denote the smallest element by $s(\cP)$ and we call a $G$-invariant Matsui prime simply \textbf{$G$-Matsui prime}.

\begin{lem} \label{sp} Suppose $\cP$ is a $G$-Matsui prime of $\cT$. Then $s(\cP)$ is also $G$-invariant, so it belongs to $Th_G(\cT)$. \end{lem}

\begin{proof} $\cP\subsetneq s(\cP)$ implies $T_{g^{-1}}(\cP)\subsetneq T_{g^{-1}} (s(\cP))$ for any $g\in G$. This in turn gives $\cP=T_{g^{-1}}(\cP)\subsetneq T_{g^{-1}} (s(\cP))$ as $\cP$ is $G$-invariant. Note that $T_{g^{-1}} (s(\cP))$ is a thick subcategory so $s(\cP)\subseteq T_{g^{-1}}(s(\cP))$ by the minimality of $s(\cP)$. Hence $T_g(s(\cP))\subseteq s(\cP)$ and this means $s(\cP)$ is $G$-invariant. \end{proof}

\begin{lem} \label{gmat} Suppose thick subcategories in $\cT^G$ are all Ind-closed. Then $\cP^G$ is a Matsui prime in $\cT^G$ if $\cP$ is a $G$-Matsui prime of $\cT$. \end{lem}

\begin{proof} Note that we have $\cP=\cP^G\cap\cT$ and $s(\cP)=s(\cP)^G\cap\cT$ because $\cP, s(\cP)$ are $G$-invariant. Suppose $\cJ\in Th(\cT^G)$ is a thick subcategory of $\cT^G$ such that $\cP^G\subsetneq\cJ$, we will show $s(\cP)^G$ is the smallest among such $\cJ$'s. First, it's obvious that $\cP^G\subsetneq s(\cP)^G$ as $\cP\subsetneq s(\cP)$ (otherwise we have $\cP=\cP^G\cap\cT=s(\cP)^G\cap\cT=s(\cP)$, a contradiction), so $s(\cP)^G$ is a thick subcategory in $\cT^G$ strictly containing $\cP^G$. Next, $\cP=\cP^G\cap\cT\subsetneq\cJ\cap\cT$ is a strict containment because $\cJ$ is Ind-closed (otherwise $\cP^G=(\cP^G\cap\cT)^G=(\cJ\cap\cT)^G=\cJ$, a contradiction). Therefore $s(\cP)\subseteq\cJ\cap\cT$ as $s(\cP)$ is the smallest such subcategories, hence $s(\cP)^G\subseteq (\cJ\cap\cT)^G\subseteq\cJ$. \end{proof}

\begin{rem} \label{ind} By Lemma \ref{ttt}, if $\cT$ is a closed tensor triangulated category generated by the tensor unit, then all thick subcategories are tt-ideals. Suppose $\cE=\cT$ is such a category and the group action is a monoidal action, one can show all thick subcategories (i.e., all tt-ideals) are Ind-closed by Lemma \ref{lem61} like in the proof of Proposition 6.5 \cite{huang2023group}. \end{rem}

Suppose $\cE=\cT$ is in fact trianglated and all thick subcategories in $\cT^G$ are Ind-closed, choose $C\subseteq Th_G(\cT)\subseteq as_G(\cT)$ to be the subset of $G$-Matsui primes in $\cT$. Then $C^G$ is contained in the subset of Matsui primes of $\cT^G$ (denote this by $D$) by Lemma \ref{gmat}. Apply Theorem \ref{actionthm} (3), we get

\begin{cor} \label{gmatimmer} Under the above hypotheses, we get an immersion \[MSpc_G(\cT)\bigg(:=Esp_C(\cT)\cong Esp_{C^G}(\cT^G)\bigg)\hookrightarrow MSpc(\cT^G)\] This is a homeomorphism if and only if all Matsui primes in $\cT^G$ is of the form $\cP^G$ for some $G$-Matsui prime $\cP$ in $\cT$. \end{cor}

\begin{ex} \label{exm} Let $G$ be a finite group acting on a Noetherian scheme $X$ over a field where $|G|\neq 0$ in the underlying field. This induces an action of $G$ on $Perf(X)$ and $Perf(X)^G\cong Perf(X/G)$, the category of perfect complexes on the quotient variety $X/G$ (e.g. Example 6.4 of \cite{huang2023group}). In addition, if $X$ is quasi-affine, $Perf(X)^G\cong Perf(X/G)$ is a triangulated category satisfies the hypotheses in Corollary \ref{gmatimmer} so there is an immersion $MSpc_G(Perf(X))\hookrightarrow MSpc(Perf(X/G))$. \end{ex}

\begin{rem} The immersion in Example \ref{exm} is in fact a homeomorphism and this is because the Matsui primes and the Balmer primes in this situation coincide. In fact, in the above situation, the Matsui spectrum is homeomorphic to the underlying scheme by Corollary 2.17 \cite{matsui2021prime}, and the underlying scheme is homeomorphic to the Balmer spectrum by result from tensor triangular geometry, so we get $MSpc_G(Perf(X))\cong BSpc_G(\cT) \cong BSpc(\cT^G)\cong MSpc(Perf(X/G))$ by Proposition \ref{equibalmer}. \end{rem}

\subsection{Group action and Serre spectrum for abelian categories} Let $\cE=\cA$ be an abelian category.

\begin{lem} Suppose $\cS\in as(\cA)$. Then \begin{enumerate} \item $\cS^G$ is a Serre subcategory of $\cA^G$, if $\cS$ is a Serre subcategory of $\cA$. \item $\cS$ is a Serre subcategory of $\cA$, if $\cS$ is $G$-invariant and $\cS^G$ is a Serre subcategory of $\cT^G$. \end{enumerate} Therefore, if $\cS$ is $G$-invariant, we have $\cS$ is Serre in $\cA$ if and only if $\cS^G$ is Serre in $\cA^G$.\end{lem}

\begin{proof} \begin{enumerate} \item Suppose $0\xra{}(A,\varphi_A)\xra{}(B,\varphi_B)\xra{}(C,\varphi_C)\xra{}0$ is a short exact sequence in $\cA^G$, then by definition this means $0\xra{}A\xra{}B\xra{}C\xra{}0$ is a short exact sequence in $\cA$. Therefore, $(B,\varphi_B)\in\cS^G\Leftrightarrow B\in\cS \Leftrightarrow A,C\in\cS$ as $\cS$ is Serre in $\cA$, and the latter by definition means $(A,\varphi_A)$ and $(C,\varphi_C)$ are in $\cS^G$. \item Let $0\xra{}A\xra{}B\xra{}C\xra{}0$ be a short exact sequence in $\cA$. Suppose $A,C\in\cS=\cS^G\cap\cA$ (because $\cS$ is $G$-invariant) so in fact we have $\bigoplus_g T_gA$ and $\bigoplus_g T_gC\in\cS$, hence $Ind(A), Ind(C)\in\cS^G$, therefore $Ind(B)\in\cS^G$. This gives $\bigoplus_g T_gB\in\cS$ so $B\cong T_eB\in\cS$ as $\cS$ is closed under direct summands. Conversely, if $B\in\cS$, we have $\bigoplus_g T_gB\in\cS$ and this gives $Ind(B)\in\cS^G$. A similar argument as above shows $A,C\in\cS$. \qedhere \end{enumerate} \end{proof}

Now, choose $C\subseteq as_G(\cA)$ to be the subset of $G$-invariant Serre subcategories in $\cA$, by Lemma \ref{gmat}, $C^G$ is contained in the subset of Serre subcategories of $\cA^G$ (denote this by $D$). Then by Theorem \ref{actionthm} (3) we have

\begin{cor} \label{gserreimmer} Let $G$ be a finite group acting on an abelian category $\cA$. Then there is an immersion \[Esp_{Serre,G}(\cA)\bigg(:=Esp_C(\cA)\cong Esp_{C^G}(\cA^G)\bigg)\hookrightarrow Esp_{Serre}(\cA^G)\] \end{cor}

\vp{0.2}

\section*{Acknowledgements}

\noindent This project is partially supported by the National Natural Science Foundation of China (Grant No.11901589), Guangdong Basic and Applied Basic Research Foundation (Grant No.2022A1515012176, No.2018A030313581) and Research Development Fund of Xi'an Jiaotong-Liverpool University (Grant No.RDF-22-02-094).

\vp{0.2}

\section*{Data availability statement}

\noindent This is a study in pure math. Data sharing not applicable to this article as no datasets were generated or analyzed during the current study.

\vp{0.2}

\bibliographystyle{plain}
\bibliography{extri}

\vp{0.1}

\noindent\emph{\footnotesize{Department of Foundational Mathematics, Xi'an Jiaotong-Liverpool University, Suzhou 215123 China, and School of Humanities and Fundamental Sciences, Shenzhen Institute of Information Technology, Shenzhen 518172 China}}

\smallskip
\smallskip
\smallskip

\noindent\emph{\footnotesize{E-mail address: xuan.yu@xjtlu.edu.cn, xuanyumath@outlook.com}}

\end{document}